\documentclass[oneside,a4paper,11pt,notitlepage]{article}
\usepackage[left=0.8in, right=0.8in, top=1.5in, bottom=1.5in]{geometry}

\usepackage{authblk}

\usepackage[T1]{fontenc} 
\usepackage[utf8]{inputenc} 
\usepackage[english]{babel}
\usepackage{lipsum} 
\usepackage{lmodern}
\usepackage{amssymb}
\usepackage{amsthm}
\usepackage{bm}
\usepackage{mathtools}
\usepackage{braket}
\usepackage{esint}
\newcommand{\abs}[1]{{\left|#1\right|}}
\newcommand{\norma}[1]{{\left\Vert#1\right\Vert}}

\usepackage{enumitem}
\usepackage{booktabs}
\usepackage{graphicx}
\usepackage{tikz}
\usetikzlibrary{patterns}
\usepackage{multicol}
\usepackage{caption}
\usepackage{varwidth}
\usepackage{lipsum}
\usepackage{appendix}
\numberwithin{equation}{section}
\usepackage[skins,theorems]{tcolorbox}
\tcbset{highlight math style={enhanced,
		colframe=black,colback=white,arc=0pt,boxrule=1pt}}
\def\XXint#1#2#3{{\setbox0=\hbox{$#1{#2#3}{\int}$}
    \vcenter{\hbox{$#2#3$}}\kern-.5\wd0}}

\theoremstyle{definition}
\newtheorem{definizione}{Definition}[section]
\theoremstyle{plain}
\newtheorem{teorema}{Theorem}[section]

\newtheorem{prop}[teorema]{Proposition}
\newtheorem{cor}[teorema]{Corollary}
\theoremstyle{definition}
\newtheorem{esempio}{Example}[section]
\newtheorem{oss}[esempio]{Remark}

\DeclareMathOperator{\R}{\mathbb{R}}

\makeatletter

\newcommand{\myfootnote}[2]{\begingroup
	\def\@makefnmark{}%
	\addtocounter{footnote}{-1}%
	\footnote{\textbf{#1} #2}
	\endgroup}
\makeatother

\usepackage{hyperref}
\hypersetup{linktoc=none, bookmarksnumbered, colorlinks=true, linkcolor=red}

\title{Hessian operators, overdetermined problems, and higher order mean curvatures: symmetry and stability results}

\date{}

\begin{document}

	\newcommand{\NN}{\mathbb{N}}
	\newcommand{\ZZ}{\mathbb{Z}}
\newcommand{\RR}{\mathbb{R}}

\newcommand{\pa}{\partial}
\newcommand{\ve}{\varepsilon}
\newcommand{\Si}{\Sigma}
\newcommand{\al}{\alpha}
\newcommand{\be}{\beta}
\newcommand{\Om}{\Omega}
\newcommand{\na}{\nabla}
\newcommand{\la}{\lambda}
	\newcommand{\La}{\Lambda}    
	\newcommand{\de}{\delta} 
    \newcommand{\te}{\theta}
    \newcommand{\nr}{\Vert}
    \newcommand{\cH}{{\mathcal H}}
\newcommand{\De}{\Delta}
\newcommand{\ol}{\overline}

\newcommand{\SerrinR}{R}


\author[1]{Nunzia Gavitone\thanks{nunzia.gavitone@unina.it}}

\author[2]{Alba Lia Masiello\thanks{masiello@altamatematica.it}}

\author[1]{Gloria Paoli\thanks{gloria.paoli@unina.it}}   

\author[3]{Giorgio Poggesi\thanks{giorgio.poggesi@adelaide.edu.au}}

\affil[1]{Dipartimento di Matematica e Applicazioni ``R. Caccioppoli'', Universit\`a degli studi di Napoli Federico II, Via Cintia, Complesso Universitario Monte S. Angelo, 80126 Napoli, Italy.}

\affil[2]{Holder of a research grant from Istituto Nazionale di Alta Matematica "Francesco Severi" at Dipartimento di Matematica e Applicazioni "R. Caccioppoli", Via Cintia, Complesso Universitario Monte S. Angelo, 80126 Napoli, Italy.}

\affil[3]{Discipline of Mathematical Sciences, The University of Adelaide, Adelaide SA 5005, Australia.}


\maketitle

{\footnotesize 
\noindent
{\it Key words and phrases.} Serrin-type overdetermined problems, Hessian operators, Alexandrov's Soap Bubble Theorem, Higher order mean curvatures, P-function, symmetry,  stability.

\noindent
{\it 2020 Mathematics Subject Classification.} 35N25, 53A10, 35G20, 35B35, 35A23.}

\begin{abstract}
{\footnotesize 
			It is well known that there is a deep connection between Serrin's symmetry result -- dealing with overdetermined problems involving the Laplacian-- and the celebrated Alexandrov's Soap Bubble Theorem (SBT) -- stating that, if the mean curvature $H$ of the boundary of a smooth bounded connected open set $\Om$ is constant, then $\Om$ must be a ball. 
			One of the main aims of the paper is to extend the study of such a connection to
            the broader case of overdetermined problems for Hessian operators and constant higher order mean curvature boundaries. Our analysis will not only provide new proofs of the higher order SBT (originally established by Alexandrov in~\cite{AlexandrovSBT,AlexandrovAMPA}) and of the symmetry for overdetermined Serrin-type problems for Hessian equations (originally established by Brandolini, Nitsch, Salani, and Trombetti in~\cite{bnst08}), but also bring several benefits, including new interesting symmetry results and quantitative stability estimates. 
			
			In fact, leveraging the analysis performed in the classical case (i.e., with classical mean curvature and classical Laplacian) by Magnanini and Poggesi in a series of papers, we will extend their approach to the higher order setting (i.e., with $k$-order mean curvature and $k$-Hessian operator, for $k \ge 1$) achieving both $L^2$-type and uniform (or Hausdorff)-type quantitative estimates of closeness to the symmetric configuration. Finally, leveraging the quantitative analysis in presence of bubbling phenomena performed in \cite{Poggesi_JMPA2025}, we also provide a quantitative stability result of closeness of almost constant $k$-mean curvature boundaries to a set given by the union of a finite number of disjoint balls of equal radii.
			
			In passing, we will also provide two alternative proofs of the result established by Brandolini, Nitsch, Salani, and Trombetti, one of which provides the extension to Hessian operators of the classical alternative proof of Serrin's symmetry result famously established by Weinberger in \cite{w71} for the classical Laplacian.
}
\end{abstract}

\section{Introduction}

\subsection{Overdetermined problems for Hessian operators}
In his pioneering paper \cite{s71}, Serrin proved  that a smooth domain $\Omega$ must be a ball if, for some constant $c >0$, there exists a solution $u \in C^2(\overline \Omega)$ to the following overdetermined problem:
\begin{equation}\label{serrin}
 \left\{
    \begin{array}{ll}
      \Delta u=1 &\text{in } \Omega\\\\
      u=0 &\text{on } \partial \Omega\\\\
      \dfrac{\partial u}{\partial \nu}=c &\text{on } \partial\Omega,
    \end{array}
  \right.
\end{equation}
where $\nu$ is the outward unit  normal to $\partial \Omega$.
The proof of this  rigidity result relies on a revisited version of the Alexandrov moving plane method and of a refinement of the maximum principle and Hopf's boundary point Lemma. 

In the same year,  Weinberger \cite{w71} provided an alternative proof of the same result in a very short way. The key ingredients of his approach  are twofold:  the $P$-function method, which ensures  that a suitable auxiliary function verifies a maximum principle, and  an integral equality, known as the Poho\v zaev identity.

Following these two seminal papers, various alternative proofs and generalizations to both linear and nonlinear operators have been developed (see, for instance, \cite{bgnt,bnst08,BNST,bnst09s,bh02,bk11,CP_IUMJ_2025,ch98,dpgx,se11,fk08,f12,fgk06,gl89,hp98,survey}). In particular in \cite{BNST},  the authors extended Serrin's  result to the $k$-Hessian operators. More precisely, they considered the following overdetermined problem:
\begin{equation}
    \label{torsionsk_intro}
   \begin{cases}
       S_k(D^2u)=\binom{n}{k} & \text{in } \Omega\\
       u=0 & \text{on } \partial\Omega\\
       |\nabla u|=R & \text{on } \partial\Omega,
   \end{cases} 
\end{equation}
where $S_k(D^2u)$, with $1\le k\le n$, denotes the $k$-th elementary symmetric function of the eigenvalues of the Hessian matrix $D^2u$ (see Section \ref{symmetric_functions} for its precise definition and properties) and $R>0$. We explicitly  observe that $S_1(D^2u)$ and $S_n(D^2u)$ correspond to the Laplacian and the Monge-Ampère operators, respectively. In particular, for $k=1$, problem \eqref{torsionsk_intro} reduces to the classical Serrin problem~\eqref{serrin}.

In \cite{BNST}, the authors established a rigidity result for problem \eqref{torsionsk_intro}:  if a solution $u \in C^2(\overline \Omega)$ to \eqref{torsionsk_intro} exists, then $u$ must be radially symmetric, and  $\Omega$ must be a ball.
Since for $k\neq 1$, $S_k(D^2 u)$ is  a fully nonlinear operator, their proof differs from those of  Serrin and Weinberger. Specifically, they rely on Newton's inequalities for the elementary symmetric functions of the eigenvalues of the Hessian matrix of $u$, combined with a Pohožaev-type identity for Hessian operators (see Section \ref{hessian_operators} for  details).

One of the aims of the present paper is to investigate the stability of problem
\eqref{torsionsk_intro}. 

For the Laplace operator case ($k=1$), the  stability of the overdeterminated Serrin problem
has been analyzed in various papers using different techniques (see, for instance, \cite{aftalion_busca,bnst08, ciraolo_magnanini_sakaguchi,poggesi,MaPogNearlyCVPDE2020,MaPog_MinE2023}). In particular, in \cite{bnst08}, the authors derived a stability result, starting from the proof of the  rigidity result contained in \cite{BNST} and obtaining   quantitative integral identities.

Moreover, in  \cite{aftalion_busca, ciraolo_magnanini_sakaguchi},  stability is studied through a quantitative analysis 
 of the moving plane method. In this regard, we also refer to the survey paper \cite{ciraolo_roncoroni_survey} and the references therein. Additionally, in \cite{poggesi}, the stability of the Serrin problem for the Laplacian is studied leveraging a proof of symmetry based on a fundamental integral identity \cite[Theorem 2.1]{poggesi} that was inspired by the approach provided in \cite[Theorems I.1, I.2]{PayneSchaefer_1989}.
 

Finally,  in \cite{bnst09s}, the authors investigated the stability issue of the overdetermined problem \eqref{torsionsk_intro} in the case $k=n$, corresponding to the Monge-Ampère operator, based on the proof of symmetry that they established in \cite{BNST}.

In this paper, we provide stability results that cover all the cases $1 \le k \le n$.

As an introductory achievement, we offer a new alternative proof of the rigidity result for \eqref{torsionsk_intro} established in \cite{BNST}, providing the extension to $k\ge 1$ of the approach famously pioneered by Weinberger in \cite{w71} for $k=1$.
This will be achieved in Theorem \ref{Serrin}; to that aim, we will consider the following auxiliary function 
\[
P=\displaystyle\frac{|\nabla u|^2}{2} \,-\, u ,
\]
where $u$ solves
\begin{equation*}
   \begin{cases}
       S_k(D^2u)=\binom{n}{k} & \text{in } \Omega\\
       u=0 & \text{on } \partial\Omega,
   \end{cases} 
\end{equation*}
and show that $P$ verifies a maximum principle (see Section \ref{P_function}).

Moreover, following the spirit of the analysis performed in \cite{poggesi} for $k=1$, we also provide various quantitative stability results.


Our main result on the quantitative stability for the problem \eqref{torsionsk_intro} is the following. 
In what follows, $P(\Om)$ denotes the perimeter of $\Om$, $r_i$ the uniform radius of the interior sphere conditions and $d_\Omega$ the diameter of $\Omega$ (we refer to Section \ref{Preliminaries} for definitions).
	\begin{teorema}\label{thm:Stability result Serrin}
		Let $n\ge 2$ and $1 \le k \le n$.
        Let $\Om \subset \RR^n$ be a $C^2$, bounded, connected, open set and let $u \in C^2 (\ol{\Om})$ be the solution to
        \begin{equation}\label{eq:Hessian Dirichlet problem}
			\begin{cases}
				S_k (D^2 u) = \binom{n}{k}  \quad &\text{ in } \Om
				\\
				u=0 \quad &\text{ on } \pa\Om .
			\end{cases}
		\end{equation}
        Let $z$ be a global minimum point of $u$ in $\ol{\Om}$ and set
		\begin{equation*}
			\SerrinR := \frac{1}{P(\Om)} \int_{\pa\Om} |\na u| d\cH^{n-1} .
		\end{equation*}
		Then, setting 
		\begin{equation}\label{eq:uniform deviation Serrin}
			\de:= \nr | \na u | - \SerrinR \nr_{L^\infty ( \pa \Om)} ,
		\end{equation} 
		we have that
		\begin{equation}\label{eq:stability theorem Serrin with rho_e - rho_i}
			\rho_e -\rho_i \le C
			\begin{cases}
				\de^{1/2}  & \quad \text{ if } n=2 ,
				\\
				\de^{1/2}  \log \left( \frac{ 1 }{ \de^{1/2}  } \right) & \quad \text{ if } n=3 ,
				\\
				\de^{\frac{1}{n-1}}  & \quad \text{ if } n \ge 4 ,
			\end{cases}
		\end{equation}
		where
		$$
		\rho_e=\max_{x\in \pa\Om}|x-z| \quad \text{ and } \quad \rho_i=\min_{x\in \pa\Om}|x-z| .
		$$

	The constant $C$ in \eqref{eq:stability theorem Serrin with rho_e - rho_i}  can be explicitly estimated only in terms of $n$, the radius $r_i$ of the uniform interior sphere condition of $\pa\Om$,
		and the diameter $d_{\Om}$ of $\Om$.
	\end{teorema}

We recall that the case $k=1$ of the above theorem was already addressed in \cite{poggesi, MaPogNearlyCVPDE2020, MaPog_MinE2023}, and, in such a particular case, finer results were established. In particular, the analysis in \cite{poggesi, MaPog_MinE2023} shows that in the case $k=1$, \eqref{eq:stability theorem Serrin with rho_e - rho_i} may be obtained with $\de = \nr | \na u | - \SerrinR  \nr_{L^1 (\pa\Om)}$, and, using $\de = \nr | \na u | - \SerrinR  \nr_{L^2 (\pa\Om)}$ one can achieve a doubling of the stability exponent (that is, \eqref{eq:stability theorem Serrin with rho_e - rho_i} with $\de^{1/2}$ replaced with $\de$); we believe that some of these improvements may be obtained also when $k > 1$, and we plan to address this in a forthcoming paper \cite{GMPP2}; in fact, in the present paper we prioritized obtaining a unified stability result (possibly not achieving the optimal rate of stability but) covering all the cases $1 \le k \le n$.

We mention that the case $k=n$ was previously studied in \cite{bnst09s} by using an alternative approach based on stability results for the isoperimetric inequality. In particular,  since  in \cite{bnst09s} the set  $\Omega$ is convex,   the constants do not depend on $r_i$. 
We stress that also our approach may be carefully adapted in the convex case to remove the dependence on $r_i$ from the constants; however, this is out of the scope of the present paper and will be addressed in forthcoming research \cite{GMPP2}.

To prove Theorem \ref{thm:Stability result Serrin}, following the spirit of \cite{poggesi}, we first establish a key integral identity -- which also provides a further short proof of the rigidity result for Problem \eqref{torsionsk_intro} (see Remark \ref{oss:rigidty again}) --, and then perform a careful quantitative analysis exploiting various tools including trace type, Sobolev-Poincar\'{e} type, and interpolation type inequalities (see the Appendix).

\subsection{Higher order Soap Bubble-type theorems for Hessian operators}

	It is well known that there is a deep connection between Serrin's symmetry result and the celebrated Alexandrov's Soap Bubble Theorem (SBT) -- stating that, if the mean curvature $H$ of the boundary of a smooth bounded connected open set $\Om\subset\mathbb{R}^n$ is constant, then $\Om$ must be a ball \cite{AlexandrovSBT,AlexandrovAMPA}. As a matter of fact, the original proof of Serrin was based on the method of the moving planes, which was inspired by the so-called reflection principle introduced by Alexandrov to prove his SBT. Moreover, this connection became more evident thanks to the alternative proof (via integral identities) of the SBT provided by Reilly \cite{Reilly_IUMJ1977,Reilly_AMM1982}, which essentially consisted in considering the torsion problem (i.e., \eqref{eq:Hessian Dirichlet problem} with $k=1$) and looking at the condition $H = constant$ as an overdetermined-type condition for this PDE; in fact, this allows to prove the SBT in a way similar to that of the alternative proof of Serrin's result famously given by Weinberger. The connection between the Serrin symmetry result and the SBT became even clearer thanks to the analysis performed in \cite{MaPoggesiJAM2019,poggesi}.
	
Our aim is to extend the study of this connection to the case $k \ge 1$, which is surprisingly still missing despite the wide literature on overdetermined problems and constant higher order mean curvature surfaces. In particular, we aim to extend to the case $k \ge 1$ the approach introduced by Reilly \cite{Reilly_IUMJ1977,Reilly_AMM1982} and further developed in \cite{MaPoggesiJAM2019,poggesi} in the classical case $k=1$. We recall that different proofs of the higher order SBT -- stating that, if the $k$-mean curvature $H_k$ of the boundary of a
smooth bounded connected open set
$\Om\subset\mathbb{R}^n$ is constant, then $\Om$ must be a ball -- can be found for instance in \cite{AlexandrovSBT,AlexandrovAMPA,MontielRos_1991,Ros_RMIb1987}. We refer to Section \ref{mean_curvatures} for the definition of $H_k$ and additional details.
    
    Our analysis will not only provide a new proof of the higher order SBT, but also bring several consequences, including a more general statement and new interesting symmetry results relating \eqref{eq:Hessian Dirichlet problem} with $H_k$. Some of these symmetry results are summarized in the following
	
	\begin{teorema}\label{thm:symmetry results summary}
		Let $\Om \subset \RR^n$ be a $C^2$, bounded, connected, open set, and let $u\in C^2(\ol{\Om})$ be the solution to \eqref{eq:Hessian Dirichlet problem}.
	Set
	\begin{equation*}
		R:= \frac{1}{P(\Om)} \int_{\pa\Om} |\na u| \, d\cH^{N-1}  \quad \text{ and } \quad \hat{R}:= \frac{n|\Om|}{P(\Om)} .
	\end{equation*}
	Then, denoting by $H_k$ the $k$-th mean curvature, $k=1,\dots, n-1$, the following  are equivalent:
	\begin{enumerate}[label=(\roman*)]
		\item $\Om$ is a ball of radius $R$, $u(x)=\frac{|x|^2 - R^2}{2}$ (up to a translation), and $ R = \hat{R}$;
		\item $| \na u |$ is constant on $\pa\Om$;
		\item $H_k \ge 1/ R^k$ on  $ \pa\Om$;
		\item $H_k \ge 1/ \hat{R}^k$ on $\pa\Om$;
		\item $|\na u|^k H_k \ge 1$ on $\pa\Om$;
		\item $| \na u | = \langle x, \nu \rangle$ on $\pa\Om$, where $\nu$ is the outward unit normal.
	\end{enumerate}
	\end{teorema}
 
	Several comments are in order. 
	
	The implication $(i) \iff (ii)$ is the Serrin-type symmetry result for Hessian equations originally proved in \cite{BNST}, of which we will provide two alternative proofs: the first by extending to $k \ge 1$ the approach initiated by Weinberger \cite{w71} for $k=1$, and the second extending to $k \ge 1$ the approach used in \cite{poggesi} for $k=1$. The last approach also gives, as a bonus, the implication $(i) \iff (vi)$, following the spirit of \cite[Theorem 2.10]{Poggesi_TAMS2024} for $k=1$.
	
	The implication $(i) \iff (iii)$ and $(i) \iff (iv)$ are new higher order soap bubble-type results that extend to $k \ge 1$ the result established in \cite[Theorem 2.2]{MaPoggesiJAM2019} for $k=1$. Notice that, for $k=1$, the divergence theorem shows that $R = \hat{R}$; we stress that, for $k \ge 1$, the implication $(i) \impliedby (iii)$ provides a more general statement than $(i) \impliedby (iv)$, as $R \ge \hat{R}$ (see Remark \ref{rem: R and hat R} for details). The classical statement of the higher order soap bubble theorem follows as an immediate corollary of the implication $(i) \impliedby (iii)$, as shown by Corollary \ref{cor:Higher order SBT with inequality}. A key step towards the proof of $(i) \impliedby (iii)$ is establishing the fundamental integral identity \eqref{eq:Fundamental identity higher oreder SBT}, which provides the extension to $k \ge 1$ of the fundamental integral identity established in \cite[Theorem 2.2]{MaPoggesiJAM2019} for $k=1$; in fact, it is easy to check that, for $k=1$, \eqref{eq:Fundamental identity higher oreder SBT} returns \cite[Identity (2.6))]{MaPoggesiJAM2019}. 
	
	The implication $(i) \iff (v)$ provides, as a particular case, the symmetry for the Hessian Dirichlet problem \eqref{eq:Hessian Dirichlet problem} under the overdetermined condition
	\begin{equation}\label{eq:INTRO:alternative overdetermination}
		|\na u| = \frac{1}{H_k^{1/k}} \quad \text{ on } \pa\Om;
	\end{equation}
	this extends to $k \ge 1$, the result established in \cite[(i) of Theorem 2.4]{MaPoggesiJAM2019}.
	We mention that, for $k=1$, this overdetermined problem has been also studied in the two-phase setting in \cite{Cavallina_IUMJ2022}.

\subsection{Stability for higher order Soap Bubble-type theorems}

   For completeness, we study the stability issue for the $k$-th order SBT, using the fine stability analysis performed in \cite{Poggesi_JMPA2025, MaPogNearlyCVPDE2020} for the classical SBT (i.e., with $k=1$) with remarkably weak $L^1$-type deviations of the form
	\begin{equation}\label{eq:intro:weak deficit for SBT alla Po and MP}
		\int_{\pa \Om} \left( \frac{1}{\hat{R}} - H \right)^+ \, d\cH^{n-1} , 
	\end{equation}
	where $\hat{R} := n|\Om|/P(\Om)$ and $\left( \frac{1}{ \hat{R} } - H \right)^+ := \max\left\lbrace 0 , \frac{1}{ \hat{R} } - H \right\rbrace$ is the positive part of the difference between the mean curvature $H$ (that is, $H_1$) of $\pa\Om$ and $1 / \hat{R}$. 
	In fact, this type of deviations crucially allows to extend the results from the classical setting (i.e., with $k=1$) to the case $k \ge 1$, leading to remarkable results with minimal effort. These results need the requirement of  $k$-convexity of the  domains (i.e., such that $H_k \ge 0$): we stress that this assumption is in the spirit of the analysis performed in \cite{CiraoloMaggi_CPAM2017} in the case $k=1$. We mention that studying the stability issue for the higher order SBT directly leveraging the fundamental identity in Theorem \ref{thm:Higher oreder SBT Fundamental Identity} may lead to alternative results that avoid the $k$-convexity restriction (as done in \cite{Poggesi_JMPA2025, MaPogNearlyCVPDE2020} for $k=1$), and this is addressed in the forthcoming paper \cite{GMPP2}.

	The following theorem can be obtained exploiting the foundamental identity presented in \cite[Theorem 3.7]{MaPogNearlyCVPDE2020} and the subsequent improvement established in \cite{MaPog_MinE2023}.
	
	\begin{teorema}\label{thm:stability k-SBT one ball}
		Let $\Om \subset \RR^n$ be a $C^2$, bounded, connected, open set with $H_k \ge 0$ and $1\leq k\leq (n-1)$. Set
		$$\hat{R}:= \frac{n|\Om|}{P(\Om)}$$ 
		and
		\begin{equation*}
			\de:= \int_{\pa \Om} \left( \frac{1}{\hat{R}^k} - H_k \right)^+ \, d\cH^{n-1} .
		\end{equation*}
		
		Then, there exists a point $z\in\Om$ such that
		\begin{equation}\label{eq:L2 stability for k-SBT}
			\nr |x-z| - \hat{R} \nr_{L^2(\pa\Om)} \le C \de^{1/2} \quad \text{ for any } n\ge 2 ,
		\end{equation}
		\begin{equation}
			\label{eq:stability HK rhoei}
			\rho_e - \rho_i  \le 
			C \,
			\begin{cases}
				\de^{1/2},  \ &\mbox{if } n =2, \, 3, 
				\\
				\de^{1/2} \max \left[ \log \left( \frac{1  }{ \de^{1/2} } \right) , 1 \right],   \ &\mbox{if } n =4 ,
				\\
				\de^{\frac{1}{N-2}},  \ &\mbox{if } n \ge 5 ,
			\end{cases}
		\end{equation}
		where we set
		$$
		\rho_e=\max_{x\in \pa\Om}|x-z| \quad \text{ and } \quad \rho_i=\min_{x\in \pa\Om}|x-z| .
		$$
		The constants $C$ in \eqref{eq:L2 stability for k-SBT} and \eqref{eq:stability HK rhoei} can be explicitly estimated only in terms of $n$, the radius $r_i$ of the uniform interior sphere condition of $\pa\Om$,
		and the diameter $d_{\Om}$ of $\Om$.
	\end{teorema}

	All the results presented so far provide closeness of $\Om$ to a single ball.
	We recall that, in general, when studying the stability for Serrin-type and Alexandrov-type symmetry results, bubbling phenomena may arise (that is, configurations of domains tending to disjoint balls of equal radii, for which the deficit $\de$ tends to zero). In the above results, the factor that prevents the bubbling phenomena is the dependence of the constants in the stability estimates on the regularity of the domain (specifically, on the radius $r_i$ of the uniform interior sphere condition of $\pa\Om$). For $k=1$, a quantitative analysis in presence of bubbling phenomena -- that is, quantitative estimates with constants not depending on parameters associated to the regularity of the domain -- was performed
	(in the classical setting) 
	in \cite{bnst08} (for Serrin's problem) and in \cite{CiraoloMaggi_CPAM2017,JulinNi_2023,Poggesi_JMPA2025} (for the SBT). In particular, \cite{Poggesi_JMPA2025} provides a fine analysis with the weak deviation \eqref{eq:intro:weak deficit for SBT alla Po and MP}, and  the use of this weak deviation allows us to leverage the results in \cite{Poggesi_JMPA2025} even in the case of higher order mean curvatures.
	In fact, the following result
	can be obtained as a consequence of \cite[Theorem 1.4]{Poggesi_JMPA2025}

    In what follows, we denote with $\omega_n$ the volume of the unit ball in $\mathbb{R}^n$.
	
	\begin{teorema}\label{thm:Stability with bubbling for k-SBT}
		Let $n \ge 2$ and $1 \le k \le n-1$. Let $\Om \subset \RR^n$ be a $C^2$, bounded, connected, open set such that
		$		H_k \ge 0 $.
		
		Set 
		\begin{equation}\label{eq:def H_0 and R}
			\hat{R}:= \frac{n|\Om|}{P(\Om)} ,
		\end{equation}
		assume that $\hat{R} \ge1$, and consider the deviation
		\begin{equation*}
			\de:= \int_{\pa \Om} \left( \frac{1}{\hat{R}^k} - H_k \right)^+ \, d\cH^{n-1} .
		\end{equation*}

		There exist $m$ points $\hat{z}_i$ such that the $m$ balls $B_{\hat{R}} (\hat{z}_i)$ (centered at $\hat{z}_i$, with radius $\hat{R}$), $i = 1, \dots, m$, are disjoint and
		\begin{equation*}
			\mathcal{U} := \bigcup_{i=1}^m B_{\hat{R}} (\hat{z}_i) ,
		\end{equation*}
		satisfies
		\begin{equation}\label{eq:bubbling 2}
			\left| \Om \De \mathcal{U} \right| \le C \,  \de^{ \frac{\al}{ \max \left\lbrace 4 , n \right\rbrace }} ,
		\end{equation}
		\begin{equation}\label{eq:bubbling 3}
			\max_{x\in \pa \mathcal{U} } \delta_{\pa\Om} (x) \le C \, \de^{\frac{\al}{4}},
		\end{equation}
		\begin{equation}\label{eq:bubbling 4}
			\left| | \pa \Om| - \left| \pa \mathcal{U} \right| \right| \le C \,  \de^{ \frac{\al}{ \max \left\lbrace 4 , n \right\rbrace }} ,
		\end{equation}
		and the number $m$ is bounded from above by means of
		\begin{equation}\label{eq:bubbling 5}
			m \le \frac{|\Om|}{\omega_n | \hat{R} - C \, \de^{\frac{\al}{4}} |^n} ;
		\end{equation}
		here, $\al:= \frac{2}{2n+7}$ and the constants $C$ appearing in \eqref{eq:bubbling 2}, \eqref{eq:bubbling 3}, \eqref{eq:bubbling 4}, and \eqref{eq:bubbling 5} can be explicitly computed and only depend on the dimension $n$ and (an upper bound on) the diameter $d_\Om$.
	\end{teorema}
Let us observe that the assumption $\hat{R} \ge 1$ is always satisfied up to a dilation.

To our knowledge, this is the first result in the literature that provides quantitative estimates of bubbling phenomena for higher order mean curvatures. We mention that a qualitative stability result was recently obtained in \cite{santilli2024finite}. Alternative quantitative estimates of closeness to a single ball (that is, in absence of bubbling phenomena) can be found in \cite{scheuer2025stability}. 


\subsection{Organization of the paper}
The rest of the paper is organized as follows.

In Section \ref{Preliminaries}, we set the notation and recall some known results about $k$-elementary symmetric functions, mean curvatures, quermassintegrals and $k$-Hessian operators. 

In Section \ref{section:overdeterminated} we deal with Serrin-type overdetermined problems for Hessian operators. Section \ref{P_function} is devoted to the study of the $P$-function and provides the extension of Weinberger's proof \cite{w71} to the broader case of Hessian operators. Section \ref{stability_issue} is devoted to the proof of Theorem \ref{thm:Stability result Serrin}, which provides quantitative stability estimates for Serrin-type overdetermined problems for Hessian operators. 

In Section \ref{section:SBT}, we study higher order soap bubble-type results and
their connection with
overdetermined problems for Hessian operators.
Section \ref{subsec:SBTandSerrin type symmetry results} is devoted to the symmetry results and contains the proof of Theorem \ref{thm:symmetry results summary}.
Section \ref{subsec:SBT stability} is devoted to the stability issue of higher order soap bubble-type results and contains the proof of Theorems \ref{thm:stability k-SBT one ball} and \ref{thm:Stability with bubbling for k-SBT}.

Finally, in the Appendix \ref{appendix}, we collect some results which are useful in the study of the quantitative stability issue, including Sobolev-Poincar\'e type and interpolation type inequalities. In particular, we prove
a Sobolev-Poincar\'{e} type inequality for solutions of uniformly elliptic equations (with explicit estimates for the constant), which may be of independent interest.

\section{Notation and preliminaries}\label{Preliminaries}
Throughout this article, we denote by $|\cdot|$, $P(\cdot)$ and $\mathcal{H}^{n-1}(\cdot)$, the $n-$dimensional Lebesgue measure, the perimeter and the $(n-1)-$dimensional Hausdorff measure in $\mathbb{R}^n$, respectively. The unit open  ball in $\mathbb{R}^n$ will be denoted by $B_1$ and  $\omega_n:=|B_1|$.
Moreover, $\de_{\pa\Om} ( \cdot )$ denotes the distance function from $\pa\Om$, and $d_\Omega$ the diameter of the set $\Omega$ (i.e., the
supremum of the distances between any two points
in the set).

In what follows,
we say that $\Omega$ has the property of uniform interior sphere condition, if there exists $r_i>0$ such that for each $x\in \partial \Omega$ there exists a ball contained in ${\Omega}$ of radius $r_i$ such that its closure intersects $\partial \Omega$ only at $x$.
\subsection{\texorpdfstring{$k$}{k}-Elementary symmetric functions } \label{symmetric_functions}

Let $A=(a_{ij})$ be a matrix in $\mathcal{S}_n$, the space of real symmetric $n\times n$ matrices, and let 
$\lambda=(\lambda_1,...,\lambda_n)$ be the vector of its eigenvalues. We define, for $k\in
\left\{1,...,n\right\}$, the $k$-th elementary symmetric function
of $A$ 
$$S_k(A)=S_k(\lambda_1,...,\lambda_n)=\sum_{1\leq i_1
<\cdots<i_k\leq n } \lambda_{i_1}\cdots \lambda_{i_k}.$$ 
By definition,  $S_k(A)$ is  the sum of all $k \times k$ principal minors of the matrix $A$.

It is well known that the
operator $S_k^{1/k}(A)$, for $k=1,...,n$, is homogeneous of degree
$1$ and  concave when restricted to the cone
$$\Gamma_k=\{A\in \mathcal{S}_n\,:\,S_i(A)\geq 0\text{ for }i=1,\dots,k\}\,.$$
Using the  notation
$$
S_k^{ij}(A)
= \frac{\partial }{\partial a_{ij}}S_k(A),$$
 Euler's identity for homogeneous functions yelds
$$
S_k(A) = \frac{1}{k} \sum_{i,j}S_k^{ij}(A) a_{ij}.
$$ 
If $A$ is a  diagonal matrix then $S_k^{i,j}$ is also diagonal and we have
\begin{equation}
\label{sij_diag}
S_{k}^{ij}(A)=\begin{cases}
    S_{k-1}(i) \quad \text{if }i=j\\ 
    0 \quad \text{if }i\neq j,
\end{cases}
\end{equation}
where $S_k(i)$ denotes the $k$-th elementary symmetric function of $\lambda$ excluding $\lambda_i$. 

Whenever $A\in\Gamma_n$, the elementary symmetric  functions satisfy the well known  Newton's inequalities:
\begin{equation}\label{newton}
    \dfrac{S_1(A)}{n}\geq\dots\geq\left(\dfrac{S_{k-1}(A)}{\binom{n}{k-1}}\right)^{1/(k-1)}\geq\left(\dfrac{S_{k}(A)}{\binom{n}{k}}\right)^{1/(k)}\geq \left(\dfrac{S_{k+1}(A)}{\binom{n}{k+1}}\right)^{1/(k+1)}\geq \dots \geq  (S_n(A))^{1/n}.
\end{equation}
\subsection{Quermassintegrals and Mean Curvatures} \label{mean_curvatures}

Let $\Omega\subset\mathbb{R}^n$ be an open, bounded set with $C^{2}$ boundary.
The $k$-th mean curvature is defined as
\begin{equation}
\label{hk}
    H_{k}=\dfrac{S_{k}(\kappa_1, \dots, \kappa_{n-1})}{\binom{n-1}{k}},
\end{equation}
where $\kappa_1, \cdots,\kappa_{n-1}$ are the principal curvautures of $\partial\Omega$ at the point $x \in \partial\Omega$.  We stress that $H_1$ and $H_{n-1}$ are the mean and the Gaussian curvature of $\partial \Omega$, respectively. We also adopt the following conventions:
\begin{equation*}
    H_0=S_0=1, \quad\quad H_n=0.
\end{equation*}
A set $\Omega$ is $k$-convex (strictly $k$-convex) for  $k\in \{1, \dots, n-1\}$ if $H_j\geq 0$ ($>0$) at every point $x\in \partial \Omega$ for $j=1,\cdots, k$.
 
The $k$-th quemassintegral of $\Omega$ is defined as
\begin{equation}
\label{wk}
    W_k(\Omega)=\dfrac{1}{n} \int_{\partial \Omega} H_{k-1} d\mathcal{H}^{n-1},
\end{equation}
with the special cases  $W_0(\Omega)=\abs{\Omega}$ and $W_1(\Omega)=P(\Omega)/n$.

We also recall  the Aleksandrov-Fenchel inequalities
\begin{equation}\label{aleksandrov-fenchel}
	\left(\dfrac{W_j(\Omega)}{\omega_n}\right)^{\frac{1}{n-j}}\geq \left(\dfrac{W_i(\Omega)}{\omega_n}\right)^{\frac{1}{n-i}},
\end{equation}
for $0\leq i<j<n$, with equality if and only if $\Omega$ is a ball. In particular, for  $i=0$, we obtain 
\begin{equation}\label{aleksandrov-fenchel_j0}
	\left(\dfrac{W_j(\Omega)}{\omega_n}\right)^{\frac{1}{n-j}}\geq \left(\dfrac{|\Omega|}{\omega_n}\right)^{\frac{1}{n}}.
\end{equation}

 Minkowski's identity  states that:
\begin{equation}\label{minkowski}
    \int_{\partial \Omega} \langle x, \nu\rangle H_kd\mathcal{H}^{n-1}=n W_k(\Omega),
\end{equation}
where $\nu$ is the outward unit normal to $\partial \Omega$.

\subsection{\texorpdfstring{$k$}{k}-Hessian Operators} \label{hessian_operators}
Let $\Om \subset \RR^n$ be a $C^2$, bounded, connected, open set and let us consider $u \in C^2(\Omega)$. We define the  \emph{$k$-Hessian operator} $S_k\left(D^2u\right)$
 as the $k$-th elementary symmetric function of $D^2u$. 

We point out that $S_k(D^2 u)$ is a second-order differential operator that reduces to the Laplace operator for $k=1$ and to the Monge-Ampère operator for $k=n$.

In general, for $k\neq 1$, $S_k(D^2u)$ is not elliptic
unless we consider the restriction to the class of $k$-convex functions (see for instance \cite{caffarelli})
$$\Phi_k^2(\Omega)=\left\{u\in C^2(\Omega)\,:\,S_i(D^2u)\geq 0 \text{
in }\Omega, i=1,2,...,k \right\}.$$ In particular 
$\Phi_n^2(\Omega)$ is equal to  the class of $C^2(\Omega)$
convex functions.

A function $u \in C^2(\Omega)$ is called strictly $k$-convex if the inequalities in the definition of the cone $\Phi_k^2(\Omega)$ hold strictly.

One can show that  $\left(S_k^{1j}(D^2u), \dots, S_k^{nj}(D^2u)\right)$ is divergence-free, allowing $S_k(D^2u)$ to be expressed in divergence form

\begin{equation}
    \label{skdiv}
    S_k(D^2u)=\frac{1}{k}\left(S_k^{ij}(D^2u) u_j\right)_i,
\end{equation}
where the indices $i, j$ denote partial derivatives. 
For convenience, when the argument of $S_k^{ij}$ is omitted, we assume  $S_k^{ij}=S_k^{ij}(D^2u)$.

If $u\in C^2(\Omega)$ and $t$ is a regular value of $u$, the $k$-Hessian operator can be related to the $k-1$-th mean curvature
on the boundary of $\{u\le t\}$ via the identity (see, for instance,  \cite{reilly73, trudinger97})  
\begin{equation}
    \label{hksk}
   \binom{n-1}{k-1} H_{k-1}=\frac{S_k^{ij}u_iu_j}{\abs{\nabla u}^{k+1}}.
\end{equation}
Moreover,
\begin{equation}\label{hksk2}
S_k(D^2 u)=\binom{n-1}{k}H_k |\nabla u|^k+\dfrac{S_k^{ij} u_i u_l u_{lj}}{|\nabla u|^2}
\end{equation}
\begin{equation}\label{S_n}
   S_n(D^2 u)=\dfrac{S_n^{ij} u_i u_l u_{lj}}{|\nabla u|^2} 
\end{equation}
Finally, we recall the Poho\v zaev identy for Hessian operators (see \cite{BNST,Tso_1990}).

\begin{prop}[Poho\v zaev identity] 
Let $\Omega$ be an open bounded set with $C^{2}$ boundary, and let $f\in C^1(\mathbb{R})$ be a non negative function. Let us define $F(u)=\int_u^0f(s)ds$. If  $u\in C^2(\Omega)\cap C^1(\overline{\Omega})$ is a  solution to
\begin{equation*}
   \begin{cases}
       S_k(D^2u)= f(u)& \text{in } \Omega\\
       u=0 & \text{on } \partial\Omega
   \end{cases}     
\end{equation*}
then, the following identity holds
\begin{equation}\label{poho_generale}
    \dfrac{n-2k}{k(k+1)} \displaystyle\int_{\Omega} S_k^{ij} u_i u_j\,dx+\dfrac{\binom{n-1}{k-1}}{k+1}\displaystyle\int_{\partial \Omega} H_{k-1}\langle x, \nu \rangle |
    \nabla u|^{k+1} d\mathcal{H}^{n-1}=n\displaystyle\int_{\Omega} F(u) \,dx
\end{equation}
\end{prop}
For  $f(u)=\binom{n}{k}$, the Poho\v zaev identy \eqref{poho_generale} simplifies to
\begin{equation}\label{poho}
    \int_{\Omega}(-u)\,dx=\frac{1}{n(n+2)} \int_{\partial \Omega}  H_{k-1}\langle x, \nu\rangle |\nabla u|^{k+1}d\mathcal{H}^{n-1}.
\end{equation}

\section{Overdetermined problems for Hessian operators}\label{section:overdeterminated}
\subsection{\texorpdfstring{$P$}{p}-function and extension of Weinberger's approach to Hessian operators}\label{P_function}
Let $\Omega \subset \R^n$ be an open, bounded, connect set with $C^{2}$ boundary, and consider $u\in C^{2}(\overline{\Omega})$ the solution to the following torsion problem for the $k$-th Hessian operator
\begin{equation}
    \label{torsion}
   \begin{cases}
       S_k(D^2u)=\binom{n}{k} & \text{in } \Omega\\
       u=0 & \text{on } \partial\Omega.
   \end{cases} 
\end{equation}
We define the linear differential operator:
\begin{equation}
\label{ell}
L[v]:=\left(S_k^{ij}(D^2u)v_j\right)_i= S_k^{ij}(D^2u)v_{ij} ,\end{equation}
and, in what follows, we consider the $P$-function:
\begin{equation}
\label{pfun}
  P=\frac{\abs{\nabla u}^2}{2}-u.  
\end{equation}
We first prove that $P$ verifies a differential inequality.
\begin{teorema}   
\label{L[P]}
Let $\Om \subset \RR^n$ be a $C^2$, bounded, connected, open set, let $u\in C^2(\overline{\Omega})$ be the solution to \eqref{torsion} and let $P$ be the function defined in \eqref{pfun}. Then,
  \begin{equation}
L[P](x) \ge 0 \quad \forall x \in \Omega, \quad  \forall k=1,\ldots,n
  \end{equation} 
  {Moreover, $L[P]=0$ if and only if $\Omega$ is a ball $B_R$ and $\displaystyle{u=\frac{\abs{x}^2-R^2}{2}}$} up to a translation.
\end{teorema}
\begin{proof}
First, let us compute $L[P]$. We have
$$\begin{aligned}
    L[P]=&\sum_{i,j=1}^n S_k^{ij}P_{ij}=
    \sum_{i,j,l} S_k^{ij}u_{li}u_{lj}+\sum_{i,j,l} S_k^{ij} u_l u_{lji} -k S_k(D^2u). 
\end{aligned}$$
Since $u$ is the solution to \eqref{torsion}, we obtain 

\begin{equation}\label{sk_constant}
    0=\frac{\partial S_k(D^2u)}{\partial x_l}= \sum_{i,j=1}^nS_k^{ij} u_{ijl}, \quad\forall l=1,\dots n,0=\frac{\partial S_k(D^2u)}{\partial x_l}= \sum_{i,j=1}^nS_k^{ij} u_{ijl}, \quad\forall l=1,\dots n,
\end{equation}
consequently,
\begin{equation}
\label{LP}
    L[P]=
        \sum_{i,j,l} S_k^{ij} u_{li}u_{lj} -kS_k(D^2u).
\end{equation}
For any $x\in \Omega$, by performing a rotation of the coordinate axes, we may assume that the Hessian matrix $D^2u$ is diagonal at the point $x$. Then, \eqref{LP} can be rewritten as:
\[
   L[P]=
        \sum_{i} S_k^{ii} u^2_{ii}-kS_k(D^2u),
\]
and, setting $\lambda_i=u_{ii}$, formula \eqref{sij_diag} implies

\begin{equation}
\label{LP2}
L[P]=\sum_{i}S_{k-1}(i)\lambda_i^2-kS_k(D^2u).
\end{equation}
Now, we must distinguish between the cases $1\le k <n$ and $k=n$.
If $1\le k<n$, using the properties of $k$-elementary symmetric functions, we get
\begin{equation*}
\frac{\partial}{\partial \lambda_i}S_{k+1}=S_k(D^2u)- S_{k-1}(i)\lambda_i.
\end{equation*}
Thus, \eqref{LP2} becomes:
\begin{equation}
\label{L}
L[P]=\Delta u \,S_k(D^2u)-(k+1)S_{k+1}(D^2u)-kS_k(D^2u).
\end{equation}
By Newton's inequalities and using the fact  that $u$ is the solution to \eqref{torsion}, we obtain:
$$
L[P](x)\ge n\binom{n}{k}-(k+1)\binom{n}{k+1}-k\binom{n}{k} =0,
$$
which proves the  claim  for $1\leq k<n$.
{Finally, to characterize the equality case for $1\le k<n$, we rewrite \eqref{L} as.
\begin{equation}
\label{L_new}
L[P]=n \binom n k \left[\left(\frac{\Delta u}{n}-\frac{S_k(D^2u)}{\binom n k}\right)+\frac{n-k}{n}\left(\frac{S_{k}(D^2u)}{\binom {n} {k}}-\frac{S_{k+1}(D^2u)}{\binom {n}{ k+1}}\right)\right].
\end{equation}

Applying  Newton's inequalities, we see that  $L[P]=0$ if and only if $D^2u$ is the identity matrix, concluding the proof in the case $1\le k<n$.} 

For $k=n$,  let us observe that $S_n(D^2u)=\text{det}(D^2u)=1$,  so $S_n^{ij}$ is the inverse matrix of $D^2u$. Therefore, 
\begin{equation}
    \label{inv1}
    \sum_j S_n^{ij}u_{jl}=\delta_{il},
\end{equation}
and 
\[
L[P]=\sum_{i}S_{n-1}(i)\lambda_i^2-nS_n(D^2u)= \Delta u \,S_n(D^2u)-nS_n(D^2u).
\]
Once again, Newton's inequality ensures that $L[P]\ge 0$ and $L[P]=0$ if and only if $D^2u$ is the identity matrix.
\end{proof}

\begin{oss}\label{ellip}
Exploiting the  regularity theory for Hessian equations, we prove that 
$L$ satisfies the uniform ellipticity condition
\begin{equation}\label{eq:explicit uniform ellipticity 1}
	\la | \xi |^2 \le  S_k^{ij}(D^2 u) \xi_i \xi_j \le \La | \xi |^2 \quad \text{ for any } x \in \Om, \, \xi \in \RR^n ,
\end{equation}
where $\la$ and $\La$ satisfy
\begin{equation}\label{eq:explicit uniform ellipticity 2}
C (n, d_\Om)^{-1} \le \la \le \La \le C (n,d_\Om) ,
\end{equation}
for a constant $C(n,d_\Om)$ only depending on $n$ and the diameter $d_\Om$ of $\Om$. 
Indeed, since $D^2 u\in \Gamma_k$,
we can arrange  its eigenvalues $\lambda=(\lambda_1,\dots,\lambda_n)$ in descending order, i.e. $\lambda_1\geq \cdots \geq \lambda_n$ and it holds, see \cite[Section $2.5$, Formula (iii) and Formula (vii)]{Wang2009}, 

$$S_{k}^{ii}(D^2u)\ge S_{k}^{11}(D^2u)$$
and 
$$S_{k}^{11}(D^2u)\ge C({n,k})\frac{S_k(D^2u)}{\lambda_1}.$$
Moreover,  by regularity theory for solution to the $k$-Hessian equation, see  \cite[Theorem 4.3]{chou}  and \cite[Theorem 3.1]{wang94}, there exists a positive constant $C$ depending only on $n$  and $\max_\Omega|u|$ such that
\begin{equation}\label{bound_w2}
\norma{D^2 u}\le C(n,\max_\Omega|u| ). 
\end{equation}
Moreover, by a comparison argument (see \cite{tso}), we have that
$$\max_\Omega|u|\le \frac{d^2_\Omega}{2} .$$
Hence, we can find a constant $C$, which only depends on $n$ and $d_\Omega$, such that
$$\lambda_1\le \norma{D^2u}\le C(n, d_\Omega).$$
On the other hand, we can use the same bound in \eqref{bound_w2} and obtain
\begin{equation}
    S_k^{ii}(D^2 u)\leq S_k^{nn}(D^2 u)=S_{k-1}(\lambda_1,\dots, \la_{n-1})\le C(n, d_\Omega).
\end{equation}

So, the  operator $L$ defined in \eqref{ell} is elliptic. 

Consequently, the maximum principle for linear elliptic operator and Theorem \ref{L[P]} imply that the function $P$ achieves its maximum on $\partial\Omega$. Moreover, the linearity of $L$ and the fact that $L[u]>0$ imply that also $\abs{\nabla u}$ achieves its maximum on $\partial\Omega$.

\end{oss}

To  prove the stability result presented in Theorem \ref{thm:Stability result Serrin}, we need the following. 
\begin{teorema} \label{h}
Let $\Om \subset \RR^n$ be a $C^2$, bounded, connected, open set. Let $u\in C^2(\overline{\Omega})$ be the solution to problem \eqref{torsion} with $1\le k \le n$. Let us define 
\begin{equation}
\label{q}
q= \frac{|x-z|^2}{2}
\end{equation}
and let $h=q-u$. Then, the following identities hold: 
\begin{itemize}
\item[(i)] $L[h]=(n-k+1)S_{k-1}(D^2 u)-k S_{k}(D^2 u)$  for $1<k\le n$ and $L[h]=0$ for $k=1;$ 
\item[(ii)]  $\displaystyle L\left[\frac{|\nabla h|^2}{2}\right]-L[h]=L[P].$
\end{itemize}
\end{teorema}
\begin{proof}
We first prove part $(i)$. By the definition of $h$, we compute:
\begin{equation}
    L[h]=\sum_{i,j}S_k^{ij}\left(
    \delta_{ij}-u_{ij}\right)=\sum_{i} S_k^{ii}-k S_k(D^2 u)= (n-k+1) S_{k-1}(D^2 u)-k S_k(D^2 u),
\end{equation}
where in the last equality we used \cite[Section 2.5, Formula (ii)]{Wang2009}:
\begin{equation}\label{wang}
    \sum_{i} S_{k+1}^{ii}=(n-k) S_{k}(D^2 u).
\end{equation}
To prove the second claim, we can differentiate $$ \dfrac{\partial }{\partial x_i}\left(\dfrac{|\nabla h|^2}{2}\right)=h_l h_{li}, \qquad \dfrac{\partial^2 }{\partial x_i\partial x_j}\left(\dfrac{|\nabla h|^2}{2}\right)=h_{lj} h_{li}+h_l h_{lji},$$
and observe that \eqref{sk_constant} implies that $$\sum_{i,j,l}S_k^{ij}h_l h_{lji}=0 ;$$ hence, we have
\begin{align*}
    \displaystyle L\left[\frac{|\nabla h|^2}{2}\right]=&\sum_{i,j,l}S_k^{ij}h_{lj}h_{li}=\sum_{i,j,l}S_k^{ij}(q_{lj}-u_{lj})(q_{li}-u_{li})=\\&=\sum_{i,j,l}S_k^{ij}q_{lj}q_{li}-S_k^{ij}u_{lj}q_{li}-S_k^{ij}q_{lj}u_{li}+S_k^{ij}u_{lj}u_{li}=\\ &
    =(n-k+1) S_{k-1}(D^2 u)-k S_k(D^2 u)-L[u]+\displaystyle L\left[\frac{|\nabla u|^2}{2}\right] =L[h]+L[P].
\end{align*}

\end{proof}

\begin{oss}
\label{rem_h0}
    Let us observe that Claim $(ii)$ in Theorem \ref{h} allows us to write

    $$\begin{aligned}
        L\left[\frac{\abs{\nabla h}^2}{2}\right]=&n \binom n k \left[\left(\frac{\Delta u}{n}-\frac{S_k(D^2u)}{\binom n k}\right)+\frac{n-k}{n}\left(\frac{S_{k}(D^2u)}{\binom {n} {k}}-\frac{S_{k+1}(D^2u)}{\binom {n}{ k+1}}\right)\right]\\
        +& k\binom{n}{k}\left[\frac{S_{k-1}(D^2 u)}{\binom{n}{k-1}}- \frac{S_k(D^2 u)}{\binom{n}{k}}\right],
    \end{aligned}$$
    so also in this case, $\displaystyle{L\left[\frac{\abs{\nabla h}^2}{2}\right]=0}$ is possible if and only if $D^2u$ is the identity matrix.
\end{oss}

As a consequence of Theorem \ref{L[P]}, we provide a new proof of the Serrin-type result for Hessian operators,
extending to $k \ge 1$ the well-known approach pioneered by Weinberger \cite{w71} in the case $k=1$.
\begin{teorema} 
\label{Serrin}
   Let $\Om \subset \RR^n$ be a $C^2$, bounded, connected, open set. Let $u \in C^2(\overline{\Omega})$ be the solution to \eqref{torsion} and  let $R>0$.  If
    $$\abs{\nabla u}=R \quad \quad \text{ on } \partial\Omega,$$
    then $\displaystyle{u=\frac{\abs{x}^2-R^2}{2}}$ and $\Omega$ is the ball $B_R$ (up to a translation).
\end{teorema}
\begin{proof}
First of all,  we observe that the existence of a solution $u\in C^2(\overline \Omega)$ ensures that $\Omega$ is strictly $(k-1)$-convex and that $u$ is strictly $k$-convex (see  \cite[Theorem 3]{caffarelli}).     
 
To prove the rigidity result,  we need to show  that $P$ is constant in $\overline \Omega$.
Since $L[P]\geq 0$ (see Theorem~\ref{L[P]}), the  maximum principle for linear elliptic operators implies that  $P$ attains its maximum over $\overline{\Omega}$
 on $\partial \Omega$. Thus, either

\smallskip

 $(i) \quad 
 \dfrac{|\nabla u|^2}{2}-u <\dfrac{R^2}{2} \text{ in } \overline\Omega,$ 
 \\
 or 

$
(ii) \quad  
\dfrac{|\nabla u|^2}{2}-u \equiv\dfrac{R^2}{2}.
$

\smallskip

Assume, by contradiction, that $(i)$ holds. Integrating over $\Omega$, we obtain 
\begin{equation*}
    \int_{\Omega}\dfrac{|\nabla u|^2}{2}\, dx-\int_{\Omega} u\, dx<\dfrac{R^2}{2}|\Omega|. 
\end{equation*}
Using Poho\v zaev's identity \eqref{poho}, Minkowski' identity \eqref{minkowski} and the overdeteminated condition, we have 
\begin{equation*}
    \int_{\Omega}(-u)\,dx= \dfrac{R^{k+1} W_{k-1}(\Omega)}{(n+2)}.\end{equation*}
Moreover, Newton's inequality implies
\begin{gather*}
\begin{split}
 \frac{R^{k+1} W_{k-1}(\Omega)}{2} =& \left(\frac{n+2}{2}\right)\int_{\Omega}( -u)\,dx\leq\dfrac{1}{2}\int_{\Omega}-u\Delta u\,dx -\int_{\Omega} u \,dx \\& = \int_{\Omega}\dfrac{|\nabla u|^2}{2}\,dx-\int_{\Omega} u\,dx<\dfrac{R^2}{2}|\Omega|,
 \end{split}
\end{gather*}
that is
    \begin{equation}\label{R_k-1}
        R^{k-1}<\dfrac{|\Omega|}{W_{k-1}(\Omega)}.
    \end{equation}
On the other hand, we have
\begin{equation*}
\begin{aligned}
  \abs{\Omega}\le\int_\Omega \frac{S_{k-1}(D^2u)}{\binom{n}{k-1}} \; dx=& \frac{1}{(k-1)\binom{n}{k-1}}\int_{\partial\Omega} S_{k-1}^{i,j} u_i \frac{u_j}{\abs{\nabla u}} \;d\mathcal{H}^{n-1}=\\& =\frac{1}{n}R^{k-1}\int_{\partial\Omega} H_{k-2} \;d\mathcal{H}^{n-1} =R^{k-1}W_{k-1}(\Omega),  
\end{aligned}
\end{equation*}
that is in contradiction with \eqref{R_k-1}, implying that $P$ must be constant and $L[P]=0$. This concludes the proof.
\end{proof}

\subsection{Quantitative stability for Serrin-type overdetermined problems with Hessian operators}\label{stability_issue}

The present section is devoted to the proof of Theorem \ref{thm:Stability result Serrin}, and provides the generalization to the case $k\ge 1$ of the approach used in \cite{poggesi} for $k=1$. 
In order to do that, we start by proving a key fundamental identity (Section \ref{key_identity}). 
\subsubsection{A Fundamental Identity \texorpdfstring{for \eqref{eq:Hessian Dirichlet problem}}{}}
\label{key_identity}
\begin{teorema}
\label{teo_main}
Let $\Om \subset \RR^n$ be a $C^2$, bounded, connected, open set. Let $u\in C^2(\overline{\Omega})$ be the solution to  problem \eqref{torsion} with $1\le k <n$. Then, we have
\begin{equation}\label{principal_equality}
\begin{aligned}
\binom n k^{-1}&\frac 2k\displaystyle\int_{\Omega}(-u)L[P]\,dx+ n\int_{\Omega}(-u)\left[\frac{\Delta u}{n}-\frac{S_{k}(D^2u)}{\binom{n}{k}}\right]\, dx = \\   \dfrac{1}{n }& \left[\int_{\partial \Omega} H_{k-1}|\nabla u|^{k+1}\left(|\nabla u|-\langle x, \nu\rangle\right)\right] \;d\mathcal{H}^{n-1}.
\end{aligned}
\end{equation}
If $k=n$, we obtain

\begin{equation}\label{principal_equality_n}
   \int_{\Omega}{(-u)L[P]}\, dx=\dfrac{1}{(n+2)}\int_{\partial \Omega}H_{n-1}|\nabla u|^{n+1}\left(\abs{\nabla u}-\langle x, \nu \rangle\right) \;d\mathcal{H}^{n-1}.\end{equation}
\end{teorema}

\begin{proof}
Let us start by computing the Dirichlet energy of $u$:
\begin{equation}\label{18}
\begin{aligned}
\int_{\Omega}\abs{\nabla u}^{2}\;dx & =\int_{\Omega}\abs{\nabla u}^{2} \frac{S_{k}\left(D^{2} u\right)}{\binom{n}{k}}\;dx  =\frac{1}{k\binom{n}{k}} \int_{\Omega}\abs{\nabla u}^{2}\left(S_{k}^{i j} u_{j}\right)_{i}  \;d\mathcal{H}^{n-1}\\
& =\frac{1}{k\binom{n}{k}}\left[-2 \int_{\Omega} S_{k}^{i j} u_{i} u_{l} u_{l j} \;dx+\int_{\partial \Omega}\abs{\nabla u} S_{k}^{i j} u_{i} u_{j} \;d\mathcal{H}^{n-1}\right]  \\
& =-\frac{2}{k\binom{n}{k}} \int_{\Omega}\left[S_{k}\left(D^{2} u\right)\abs{\nabla u}^{2}\; -\binom{n-1}{k}H_{k}\abs{\nabla u}^{k+2} \right]dx\\ &+\frac{1}{n} \int_{\partial \Omega} H_{k-1} \abs{\nabla u}^{k+2}\, d\mathcal{H}^{n-1},
\end{aligned}
\end{equation}
where the last equality follows from formula \eqref{hksk2}.
Moreover,  identity \eqref{hksk} gives
\begin{equation}\label{19}
   \binom{n-1}{k} \int_\Omega H_k \abs{\nabla u}^{k+2}\, dx = (k+1) \int_\Omega (-u) S_{k+1}(D^2u)\, dx.
\end{equation}
Combinig \eqref{18} and \eqref{19}, we obtain
\begin{equation*}
    \int_\Omega (-u) \frac{S_{k+1}(D^2u)}{\binom{n}{k+1}}\, dx= \frac{(k+2)}{2(n-k)}\int_\Omega \abs{\nabla u}^2\, dx - \frac{k}{2n(n-k)}\int_{\partial \Omega} H_{k-1} \abs{\nabla u}^{k+2}\, d\mathcal{H}^{n-1}.
\end{equation*}
By  the  Poho\v zaev identity for Hessian equations \eqref{poho}, we have
\begin{equation*}
    \int_{\Omega}(-u)\; dx=\frac{1}{n(n+2)} \int_{\partial \Omega}  H_{k-1}\langle x, \nu\rangle |\nabla u|^{k+1} \;d\mathcal{H}^{n-1}.
\end{equation*}
Then, we get
\begin{equation*}
    \begin{aligned}
     &\int_{\Omega}(-u)\left[\frac{S_k(D^2u)}{\binom{n}{k}}-\frac{S_{k+1}(D^2u)}{\binom{n}{k+1}}\right]\, dx= \frac{1}{n(n+2)}\int_{\partial\Omega}H_{k-1}\abs{\nabla u}^{k+1} \langle x,\nu\rangle \;d\mathcal{H}^{n-1}+\\
     -&\frac{(k+2)}{2(n-k)}\int_\Omega \abs{\nabla u}^2\, dx + \frac{k}{2n(n-k)}\int_{\partial \Omega} H_{k-1} \abs{\nabla u}^{k+2}\, d\mathcal{H}^{n-1}=\\
     &\frac{n(k+2)}{2(n-k)}\int_\Omega(-u)\left[\frac{S_k(D^2u)}{\binom{n}{k}}-\frac{\Delta u}{n}\right] dx+ \frac{k}{2n(n-k)}\int_{\partial\Omega} H_{k-1}\abs{\nabla u}^{k+1}\left(\abs{\nabla u}-\langle x , \nu \rangle\right) \;d\mathcal{H}^{n-1},
    \end{aligned}
\end{equation*}

Finally, by \eqref{L_new}, we obtain the claim. 

In order to obtain \eqref{principal_equality_n}, we procced as before, using \eqref{S_n}.
\end{proof}
As a consequence of the fundamental identity \eqref{principal_equality} and \eqref{principal_equality_n}, we obtain the following inequality, which will be the starting point in proving our stability result.
\begin{prop}\label{44} Let $\Om \subset \RR^n$ be a $C^2$, bounded, connected, open set. Let $u\in C^2(\overline{\Omega})$ be the solution to problem \eqref{torsion} with $1\leq  k \leq n$ and $M=\max_{\overline \Omega}|\nabla u|=\max_{\partial \Omega}|\nabla u|$. Then, 
\begin{equation}\label{dish}
\begin{aligned}
\binom n k^{-1}\frac 2k\displaystyle\int_{\Omega}(-u)L\left[\frac{\abs{\nabla h}^2}{2}\right]\,dx\le &  \dfrac{1}{n } \left[\int_{\partial \Omega} H_{k-1}|\nabla u|^{k+1}\left(|\nabla u|-\langle x, \nu\rangle\right)\right] \,d \mathcal{H}^{n-1}\,+\\
&\frac{M^2}{n}\,\int_{\partial \Omega} \left(H_{k-2} |\nabla u|^{k-1}-<x,\nu>\right)\,d\mathcal{H}^{n-1}
\end{aligned}
\end{equation}
\end{prop}
\begin{proof}
We add to both sides of \eqref{principal_equality} the following quantity:
\[
\binom{n}{k}^{-1}\frac{2}{k} \int_{\Omega}(-u)L[h]\, dx=\displaystyle 2\int_{\Omega}(-u) \left[ \frac{S_{k-1}(D^2 u)}{\binom{n}{k-1}}-1\right]\,dx. 
\]
By the maximum principle for $P$, we get
\begin{equation}
\label{max1}
\max_{\overline \Omega}(-u) <\frac{M^2}{2},
\end{equation}
hence, by \eqref{max1} and the divergence Theorem
\[
\begin{aligned}
\binom{n}{k}^{-1}\frac{2}{k} &\int_{\Omega}(-u)L[h]\, dx\le\displaystyle M^2\int_{\Omega} \left[ \frac{S_{k-1}(D^2 u)}{\binom{n}{k-1}}-1\right]\,dx =\\
\frac{M^2}{n}&\int_{\partial\Omega}\left(H_{k-2} |\nabla u|^{k-1}-<x,\nu>\right)\,d\mathcal{H}^{n-1},
\end{aligned}
\]
that is the desired inequality.
\end{proof}

We conclude with the following remark, where we observe that as a consequence of Proposition~\ref{44} we can derive the rigidity result stated in Theorem \ref{Serrin}.
\begin{oss}[An alternative proof of Theorem \ref{Serrin}]\label{oss:rigidty again}

Starting from formula \eqref{dish}, we  add and subtract to the left hand side  the term
$$\frac{R^{k+1}}{n}\int_{\partial\Omega}H_{k-1}\left(|\nabla u|-\langle x, \nu\rangle\right)\, d\mathcal{H}^{n-1},$$
obtaining 
\begin{equation}\label{full}
\begin{aligned}
&\binom n k^{-1}\frac 2k\displaystyle\int_{\Omega}(-u)L\left[\frac{\abs{\nabla h}^2}{2}\right]\,dx 
\le    \dfrac{1}{n } \int_{\partial \Omega} H_{k-1}\left(|\nabla u|^{k+1}-R^{k+1}\right)\left(|\nabla u|-\langle x, \nu\rangle\right)\, d\mathcal{H}^{n-1}+\\ &+ \frac{R^{k+1}}{n}\int_{\partial\Omega}H_{k-1}\left(|\nabla u|-\langle x, \nu\rangle\right)+ \frac{M^2}{n}\,\int_{\partial \Omega} \left(H_{k-2} |\nabla u|^{k-1}-<x,\nu>\right)\,d\mathcal H^{n-1}.
\end{aligned}
\end{equation}
    Now, the overdeterminated condition  $\abs{\nabla u}=R$ on $\partial\Omega$ implies
\begin{equation}
\label{come}
\begin{aligned}
\binom n k^{-1}\frac 2k\displaystyle\int_{\Omega}(-u)L\left[\frac{\abs{\nabla h}^2}{2}\right]\,dx 
\le &  {R^{k+1}}(RW_k(\Omega)-W_{k-1}(\Omega))+ {M^2} R^{k-1}W_{k-1}(\Omega)-M^2 \abs{\Omega}\\
=& R^{k+2}W_k(\Omega)- R^2\abs{\Omega},
\end{aligned}
\end{equation}
where we use the fact that
$$M=\max_{\Omega}\abs{\nabla u}=\max_{\partial\Omega}\abs{\nabla u}=R.$$
On the other hand, we have by Newton's inequality,
\begin{equation}
\label{RW}
    \abs{\Omega}=\int_\Omega \frac{S_k(D^2u)}{\binom{n}{k}} \; dx= \frac{1}{k\binom{n}{k}}\int_{\partial\Omega} S_k^{i,j} u_i \frac{u_j}{\abs{\nabla u}} \;d\mathcal{H}^{n-1}= \frac{1}{n}R^k\int_{\partial\Omega} H_{k-1}  \;d\mathcal{H}^{n-1}=R^k W_{k}(\Omega).
\end{equation}
One can reobtain the rigidity result combining  \eqref{come}, \eqref{RW} and Remark \ref{rem_h0}, proving  that, whenever the overdeterminated condition holds, we have that $\displaystyle{u=\frac{\abs{x}^2-R^2}{2}}$ and $\Omega$ is the ball $B_R$ up to a translation.
\end{oss}

\subsubsection{Proof of Theorem \ref{thm:Stability result Serrin}}
\label{serrin_stabilty}

In the following proposition, we bound the right-hand side of \eqref{dish} from above in terms of the deviation $\de$ defined in \eqref{eq:uniform deviation Serrin}.

\begin{prop}\label{step1}
    Let $\Om \subset \RR^n$ be a $C^2$, bounded, connected, open set, and let $u\in C^2(\overline{\Omega})$ be the solution to   problem  \eqref{torsion}
with $1\leq  k \leq n$. Let $\delta$ be defined as \eqref{eq:uniform deviation Serrin}. 
Then,  $\Omega$ is strictly $(k-1)$-convex, $u$ is strictly $k$-convex and

\begin{equation}
\label{final_step1}
\begin{aligned}
    \dfrac{1}{n } &\left[\int_{\partial \Omega} H_{k-1}|\nabla u|^{k+1}\left(|\nabla u|-\langle x, \nu\rangle\right)\right] \,d \mathcal{H}^{n-1}\,+\frac{M^2}{n}\,\int_{\partial \Omega} \left(H_{k-2} |\nabla u|^{k-1}-<x,\nu>\right)\,d\mathcal{H}^{n-1}\\ \le& W_k(\Omega) R^k \left[2R+d_\Omega(k+2)\right] \delta  .     
\end{aligned}
\end{equation}
\end{prop}
\begin{proof}
First of all  let us  observe that the existence of a solution $u\in C^2(\overline \Omega)$ ensures that $\Omega$ is strictly $(k-1)$-convex and that $u$ is strictly $k$-convex (see  \cite[Theorem 3]{caffarelli}).

We observe that \eqref{eq:uniform deviation Serrin} implies
\begin{equation}
\begin{aligned}
   \abs{\Omega}\ge(R-\delta)^k W_{k}(\Omega)=(R-\delta)^k\frac{1}{n}\int_{\partial\Omega} H_{k-1} \;d\mathcal{H}^{n-1},\\
  \abs{\Omega}\le (R+\delta)^k\frac{1}{n}\int_{\partial\Omega} H_{k-1} \;d\mathcal{H}^{n-1} =(R+\delta)^k W_{k}(\Omega).
  \end{aligned}
\end{equation}
  We now  aim to bound in terms of $\delta$ the right-hand side of \eqref{dish}, that is
    $$ \dfrac{1}{n } \left[\int_{\partial \Omega} H_{k-1}|\nabla u|^{k+1}\left(|\nabla u|-\langle x, \nu\rangle\right)\right] \,d \mathcal{H}^{n-1}\,+ \frac{M^2}{n}\,\int_{\partial \Omega} \left(H_{k-2} |\nabla u|^{k-1}-<x,\nu>\right)\,d\mathcal{H}^{n-1}.$$
      Since $M\le R+\delta$, the second term can be bounded as follows 
\begin{equation}
     \frac{M^2}{n}\,\int_{\partial \Omega} \left(H_{k-2} |\nabla u|^{k-1}-<x,\nu>\right)\,d\mathcal{H}^{n-1}\le (R+\delta)^{k+1}W_{k-1}(\Omega)- (R+\delta)^2\abs{\Omega}.
\end{equation}
As for the first term, let us observe that
     $$
            \frac{1}{n}\int_{\partial\Omega}H_{k-1}\abs{\nabla u}^{k+2}\, d\mathcal{H}^{n-1}\le (R+\delta)^{k+2} W_k(\Omega), 
    $$
    while

$$\begin{aligned}
            \frac{1}{n}\int_{\partial\Omega}H_{k-1}\abs{\nabla u}^{k+1}\langle x, \nu\rangle\, d\mathcal{H}^{n-1}= \frac{1}{n}\int_{\partial\Omega}H_{k-1}\langle x, \nu\rangle\left(\abs{\nabla u}^{k+1}-R^{k+1}\right)\, d\mathcal{H}^{n-1} +R^{k+1}W_{k-1}(\Omega).
            \end{aligned}$$
So, by \eqref{eq:uniform deviation Serrin}, we have
    $$\begin{aligned}
           \abs{ \frac{1}{n}\int_{\partial\Omega}H_{k-1}\abs{\nabla u}^{k+1}\langle x, \nu\rangle\, d\mathcal{H}^{n-1}- R^{k+1}W_{k-1}(\Omega)}\leq (k+1) d_\Omega W_k(\Omega) R^k \delta 
            \end{aligned}.$$
 Hence, the right-hand side of \eqref{dish} can be bounded from above in terms of $\delta$ as follows: 
\begin{equation*}
   \begin{aligned}
  & \dfrac{1}{n } \left[\int_{\partial \Omega} H_{k-1}|\nabla u|^{k+1}\left(|\nabla u|-\langle x, \nu\rangle\right)\right] \,d \mathcal{H}^{n-1}\,+ \frac{M^2}{n}\,\int_{\partial \Omega} \left(H_{k-2} |\nabla u|^{k-1}-<x,\nu>\right)\,d\mathcal{H}^{n-1}\leq \\
     &(R+\delta)^{k+2} W_k(\Omega)- R^{k+1}W_{k-1}(\Omega) +(k+1) d_\Omega W_k(\Omega) R^k \delta +\\&+(R+\delta)^2\left[(R+\delta)^{k-1}W_{k-1}(\Omega)-\abs{\Omega}\right]\le\\&
         W_k(\Omega) [(R+\delta)^{k+2}- (R+\delta)^2(R-\delta)^k]+(k+1) d_\Omega W_k(\Omega) R^k \delta+\\&+ W_{k-1}(\Omega)[(R+\delta)^{k+1}-R^{k+1}]
       \le  W_k(\Omega) R^k \left[2R+d_\Omega(k+2)\right] \delta.
       \end{aligned}
   \end{equation*}

\end{proof}

We are now in position to prove Theorem \ref{thm:Stability result Serrin}.

\begin{proof}[Proof of Theorem \ref{thm:Stability result Serrin}]

In what follows, we will denote the constants with the letter $C$ and specify between round brackets the parameters on which they depend. The value of $C$ may change from line to line and may be explicitly estimated in terms of the indicated parameters by following the steps of the proof.

{\bf Step 1 (Preliminary estimates).}
Notice that it is enough to prove the theorem for $\de \le \de_0$, where $\de_0=\de_0(n, r_i, d_\Om )$ is some explicit positive constant only depending on $n, r_i, d_\Om$. In fact, if $\de \ge \de_0$, then the statement trivially holds true, as
\begin{equation*}
	\rho_e - \rho_i \le d_\Om \le \frac{d_\Om}{\de_0^\tau} \de^\tau 
\end{equation*}
for any $\tau >0$.
Thus, in what follows we can assume that
\begin{equation*}
	\de \le 1
\end{equation*}
so that by \eqref{eq:uniform deviation Serrin}, the following gradient bound holds true
\begin{equation}\label{eq:Gradient bound for u with uniform deviation}
	M:= \nr \na u \nr_{L^\infty (\Om)} = \max_{\ol{\Om}} |\na u| = \max_{\pa{\Om}} |\na u| \le \SerrinR + \de
	\le \SerrinR +1  
	\le C(n,d_\Om),
\end{equation}
for some explicit constant $C(n,d_\Om)$. 
The last inequality in \eqref{eq:Gradient bound for u with uniform deviation} easily follows using that
\begin{equation}\label{eq: explicit bound for Rmedia with de le 1}
	R \le C(n, d_\Om) , 
\end{equation}
for some explicit constant $C(n, d_\Om)$, which can be obtained as follows.
Using the divergence theorem and recalling \eqref{eq:Hessian Dirichlet problem}, we compute
\begin{equation}\label{eq:1 stimaRmedia}
	\int_{\pa\Om} \left(S_k^{ij} (D^2 u) u_i \frac{u_j}{| \na u|} \right) \, d \cH^{n-1} = \int_\Om S_k(D^2u) \,dx = \binom{n}{k} |\Om| .
\end{equation}
Recalling now Remark \ref{ellip},
we have that
\begin{equation}\label{eq:2 stimaRmedia}
	\int_{\pa\Om} | \na u | d\cH^{n-1} \le C(n, d_\Om)   \int_{\pa\Om}  \left[S_k^{ij}(D^2 u) u_j \frac{u_i}{|\na u|} \right] d\cH^{n-1},
\end{equation}
whereas from the definition of $R$ and the isoperimetric inequality we have that
\begin{equation}\label{eq:3 stimaRmedia}
n\omega_n^{\frac{1}{n}} |\Om|^{1-\frac{1}{n}} R \le P(\Om) R = \int_{\pa\Om} | \na u | d\cH^{n-1} .
\end{equation}
Combining \eqref{eq:1 stimaRmedia}, \eqref{eq:2 stimaRmedia}, \eqref{eq:3 stimaRmedia}, and using the isodiametric inequality
\begin{equation*}
	|\Om| \le \omega_n \left( \frac{d_\Om}{2}\right)^n,
\end{equation*}
we find an explicit constant $C(n,d_\Om)$ such that \eqref{eq: explicit bound for Rmedia with de le 1} holds true.

Thanks to the comparison principle for the Hessian operator, comparing the solution $u$ to \eqref{eq:Hessian Dirichlet problem} and
$$
w(y):= \frac{\abs{y-x}^2- \de_{\pa\Om}^2(x)}{2} ,
$$
that is the solution to \eqref{eq:Hessian Dirichlet problem} in $B_{\de_{\pa\Om}(x)} (x) \subset \Om$,
we can see that
\begin{equation}\label{distanza}
       \frac{1}{2}\de_{\pa\Om}^2(x) \le   -u(x), \quad \text{ for any } x\in \overline{\Omega} ,
    \end{equation}
and, since $\Om$ satisfies the $r_i$-uniform interior sphere condition, we can also claim that,
\begin{equation}\label{eq:distance and u relation: Lipschitz}
\frac{r_i}{2} \de_{\pa\Om}(x) \le 	-u(x) \quad \text{ for any } x \in \ol{\Om} .
\end{equation}
Let us prove \eqref{eq:distance and u relation: Lipschitz}: if $\delta_{\pa \Om}(x)\ge r_i$, the claim follows trivially, while if $\delta_{\pa \Om}(x)< r_i$, let $y$ be the closest point in $\pa \Omega$ to $x$ and call $B$ the ball of radius $r_i$ touching $\pa \Om$ at $y$ and containing $x$. Up to a translation, we can assume that the center of $B$ is the origin. If $w$ is the solution to \eqref{eq:Hessian Dirichlet problem} on $B$  $$w(s):= \frac{\abs{s}^2- {r_i}^2}{2}, $$
then, by comparison we have that $w\ge u$ in $B$ and so 
\begin{equation*}
    -u(x)\ge \frac{r_i^2 - \abs{x}^2 }{2}=\frac{ (r_i+\abs{x}) }{2}  (r_i-\abs{x} ) \ge \frac{ r_i }{2} (r_i-\abs{x}),
\end{equation*}
which implies \eqref{eq:distance and u relation: Lipschitz}, since $r_i-\abs{x}\ge\delta_{\partial \Omega}(x)$.

{\bf Step 2 (Linking $\int_{\Omega}(-u) |D^2 h|^2\, dx$ and the right-hand side of \eqref{dish}).}  By the ellipticity of the operator $L$ (see Remark \ref{ellip}), it holds
\begin{equation}\label{eq:relation FrobeniusnormHessian and L}
|D^2 h|^2 \leq  C(n, d_\Om) \,  L\left[\frac{\abs{\nabla h}^2}{2}\right].
\end{equation}

Combining \eqref{dish}, \eqref{final_step1},
\eqref{eq: explicit bound for Rmedia with de le 1}, and \eqref{eq:relation FrobeniusnormHessian and L} we thus obtain that
\begin{equation}\label{eq:weighted Hessian and RHS}
  \int_{\Omega}(-u) |D^2 h|^2\, dx \le C(n, r_i ,d_\Om ) \de .
\end{equation}

\bigskip

{\bf Step 3 (Uniform stability).}
Let $z$ be a global maximum point of $-u$ in $\ol{\Om}$; notice that $z\in\Om$. Set 
\begin{equation*}
    q = \frac{|x-z|^2}{2} \quad \text{ and }  h= q-u ,
\end{equation*}
\begin{equation*}
\rho_e=\max_{x\in \pa\Om}|x-z| \quad \text{ and } \quad \rho_i=\min_{x\in \pa\Om}|x-z| .
\end{equation*}
Thus, since $u=0$ on $\pa\Om$, the definitions of $h$, $q$, $\rho_i$, and $\rho_e$ give
\begin{equation*}
	\max_{\ol{\Om}} h - \min_{\ol{\Om}} h \ge 	\max_{\pa\Om} h - \min_{\pa\Om} h = 
	\max_{\pa\Om} q - \min_{\pa\Om} q = \frac{1}{2} \left( \rho_e^2 - \rho_i^2 \right) \ge \frac{r_i}{2} \left( \rho_e-\rho_i \right),
\end{equation*}
that is,
\begin{equation*}
	\rho_e -\rho_i \le  \frac{2}{r_i} \left( \max_{\ol{\Om}} h - \min_{\ol{\Om}} h \right).
\end{equation*}

Combining this with
Theorem \ref{thm:Interpolation from MPMine + MPCVPDE}  (with $f:=h$, and $q:=\infty$), and with (iii) of Remark~\ref{rem:(iii)if r_i then b_0, theta, a, and C_p can be estimated (iv) for Poincare with L and r_i.} in the Appendix, we find that
\begin{equation}\label{eq:interpolating inequality for h}
\rho_e -\rho_i \le C(n, p, r_i, d_\Om)
\begin{cases}
\nr \na h \nr_{L^p}(\Om)  & \quad \text{ if } p>n ,
\\
\nr \na h \nr_{L^n(\Om)} \log \left(e \,  \frac{ |\Om|^{\frac{1}{n}} M}{ \nr \na h \nr_{L^n(\Om)}}  \right)  & \quad \text{ if } p=n ,
\\
M^{\frac{n-p}{n}} \nr \na h \nr_{L^p(\Om)}^{\frac{p}{n}}  & \quad \text{ if } 1 \le p < n ,
\end{cases}
\end{equation}
where $C$ only depends on $n, p, d_\Om$ and $r_i$, and, as usual, $M:= \nr \na u \nr_{L^\infty(\Om)} = \max_{\ol{\Om}}|\na u|$ can be estimated by means of \eqref{eq:Gradient bound for u with uniform deviation}.

Now we recall that we chose $z$ as a global maximum point of $-u$ in $\ol{\Om}$.
With this choice we have that $z\in\Om$ and $h_i(z)=0$ and $L[h_i]=0$. Moreover, we claim that
\begin{equation}\label{eq:lower bound distance z}
	\de_{\pa\Om} (z) \ge \frac{r_i^2}{M} ,
\end{equation}
where $r_i$ is the radius of the uniform interior sphere condition and $M:=\max_{\ol{\Om}}|\na u|$ (which, as already noticed, can be estimated by means of \eqref{eq:Gradient bound for u with uniform deviation}).
We mention that various finer lower bounds for $\de_{\pa\Om} (z)$ may be obtained (see \cite{MaPog_HotSpotsMA2022}), but  \eqref{eq:lower bound distance z} is precisely suited to the parameters considered in this paper. 

To prove \eqref{eq:lower bound distance z}, we choose $x_z \in \pa\Om$ such that $\de_{\pa\Om}(z)=|x_z - z|$, $y\in\Om$ such that $\de_{\pa\Om} (y)=r_i $, and compute
\begin{equation*}
	\frac{r_i^2}{2} \le - u(y) \le \max_{\ol{\Om}} (-u) = -u(z) = u(x_z) -u(z) \le M \, \de_{\pa\Om} (z) ,
\end{equation*}
where in the first inequality we used \eqref{distanza}.

Using Theorem \ref{thm:Poincare for solutions} (with $f:=h_i$ for $i=1,\dots,n$, $x_0 := z$ and $\mathcal{L}:=L$) and item (iv) of Remark~\ref{rem:(iii)if r_i then b_0, theta, a, and C_p can be estimated (iv) for Poincare with L and r_i.} in the Appendix, 
\eqref{eq:lower bound distance z}, and \eqref{eq:Gradient bound for u with uniform deviation}, we thus find that the inequality
\begin{equation}\label{eq:Poincare su gradiente di h}
	\nr \na h \nr_{L^r (\Om) } \le C ( r, n, r_i, d_\Om ) \, \nr \de_{\pa\Om}^{\frac{1}{2}} D^2 h \nr_{L^2(\Om)} 
\end{equation}
holds true for any $r$ such that $2 \le r \le \frac{2n}{n-1}$ with a constant $C ( r, n, r_i, d_\Om )$ that only depends on $r$, $n$, $r_i$, $d_\Om$.

Hence, combining \eqref{eq:Poincare su gradiente di h} and \eqref{eq:distance and u relation: Lipschitz} we obtain that
\begin{equation}\label{eq:PRONTA Poincare su gradiente di h con peso (-u)}
	\nr \na h \nr_{L^r (\Om) } \le C(r, n, r_i, d_\Om) \, \nr (-u)^{\frac{1}{2}} D^2 h \nr_{L^2(\Om)} 
\end{equation}
holds true for any $r$ such that $2 \le r \le \frac{2n}{n-1}$ with a constant $C(r, n, r_i, d_\Om)$ that only depends on $r$, $n$, $r_i$, $d_\Om$.

We now distinguish three cases.

(i) When $n=2$, combining \eqref{eq:interpolating inequality for h} with $p:=4$ and \eqref{eq:PRONTA Poincare su gradiente di h con peso (-u)} with $r:= 4$ gives that
\begin{equation}\label{eq:in proof case n=2}
	\rho_e -\rho_i \le C (n , r_i , d_\Om) \nr (-u)^{\frac{1}{2}} D^2 h \nr_{L^2 (\Om) } .
\end{equation}

(ii) When $n=3$, combining \eqref{eq:interpolating inequality for h} with $p:=n=3$ and \eqref{eq:PRONTA Poincare su gradiente di h con peso (-u)} with $r:= 3$ gives that
\begin{equation}\label{eq:in proof case n=3}
	\rho_e -\rho_i \le C (n , r_i , d_\Om) \nr (-u)^{\frac{1}{2}} D^2 h \nr_{L^2 (\Om)} \max \left[ \log \left( \frac{ M }{ \nr (-u)^{\frac{1}{2}} D^2 h \nr_{L^2 (\Om)}}    \right) , 1  \right].
\end{equation}
Here, we used that the function $t \to t \max \left[ \log \left( \frac{A}{t} \right) , 1  \right]$ is non-decreasing for any $A>0$.

(iii) When $n\ge 4$, combining \eqref{eq:interpolating inequality for h} with $p:= 2n/(n-1)$ and \eqref{eq:PRONTA Poincare su gradiente di h con peso (-u)} with $r:= 2n/(n-1)$ gives
\begin{equation}\label{eq:in proof case n>=4}
	\rho_e -\rho_i \le C (n , r_i , d_\Om) M^{\frac{n-3}{n-1}} \nr (-u)^{\frac{1}{2}} D^2 h \nr_{L^2(\Om)}^{\frac{2}{n-1}}.
\end{equation}

The desired conclusion in \eqref{eq:stability theorem Serrin with rho_e - rho_i} follows by combining \eqref{eq:in proof case n=2}, \eqref{eq:in proof case n=3}, and \eqref{eq:in proof case n>=4}, with the estimate for $M$ in \eqref{eq:Gradient bound for u with uniform deviation} and the inequality
\begin{equation*}
	\nr (-u)^{\frac{1}{2}} D^2 h \nr_{L^2(\Om)} \le C( n, r_i, d_\Om ) \,  \de^{1/2} ,
\end{equation*}
which follows from \eqref{eq:weighted Hessian and RHS}.

\bigskip

\end{proof}

\section{Higher order Soap Bubble-type theorems for Hessian operators}\label{section:SBT}

\subsection{Symmetry results and proof of Theorem \ref{thm:symmetry results summary}}\label{subsec:SBTandSerrin type symmetry results}
	
	We start by establishing the following fundamental integral identity, which extends to $k \ge 1$ that established in \cite[Identity (2.6))]{MaPoggesiJAM2019} for $k=1$.
	
	\begin{teorema}\label{thm:Higher oreder SBT Fundamental Identity}
	Let $\Om \subset \RR^n$ be a $C^2$, bounded, connected, open set and let $u\in C^2(\ol{\Om})$ be a solution of \eqref{eq:Hessian Dirichlet problem}.
	
	Then, setting 
	\begin{equation}\label{def: R new definition}
		R:= \frac{1}{P(\Om)} \int_{\pa\Om} |\na u| \, d\cH^{N-1} ,
	\end{equation}
	the following integral identity holds true:
	\begin{equation}\label{eq:Fundamental identity higher oreder SBT}
	\begin{split}
	\binom{n}{k}^{-1} \int_\Om L[P] \, dx
	+ k \, \widetilde{\mathcal{D}}_1 + \widetilde{\mathcal{D}}_2
	= 
	\int_{\pa\Om} \left( \frac{1}{R^k} - H_{k} \right) | \na u|^{k+1} \, d\cH^{n-1} ,
	\end{split}
	\end{equation}
	where
	\begin{equation*}
		\widetilde{\mathcal{D}}_1:=  \int_\Om \left[ 1 - \frac{S_{k+1}(D^2 u)}{\binom{n}{k+1}} \right] \, dx ,   
	\end{equation*}
	\begin{equation*}
	\widetilde{\mathcal{D}}_2:=   \frac{1}{R^k} \int_{\pa\Om} | \na u|^{k+1} \, d\cH^{n-1} - 	 \int_{\pa\Om} |\na u| d\cH^{n-1}   . 
	\end{equation*}
	\end{teorema}

	\begin{oss}\label{rem:Positivity of deficits}
	{\rm
	$(i)$ Recalling Theorem \ref{L[P]} we have that $L[P] \ge 0$ in $\Om$, and $L[P] \equiv 0$ if and only if $\Om$ is a ball. Thanks to this feature, $L[P]$ will be used as a {\it spherical detector}.
		
	$(ii)$ Recalling Newton's inequality \eqref{newton} and the equation $S_k(D^2u)= \binom{n}{k}$, it is immediate to check that
	\begin{equation*}
	\widetilde{\mathcal{D}}_1 = \int_\Om \left[  \left( \frac{S_{k}(D^2 u)}{\binom{n}{k}} \right)^{\frac{k+1}{k}} - \frac{S_{k+1}(D^2 u)}{\binom{n}{k+1}} \right] \, dx  \ge 0.
	\end{equation*}
		
	$(iii)$ By using the H\"older inequality and the definition of $R$ in \eqref{def: R new definition}, it is easy to check that
	\begin{equation*}
	 \int_{\pa\Om} |\na u| d\cH^{n-1}   \le \frac{1}{R^{k}} \int_{\pa\Om} |\na u|^{k+1} d\cH^{n-1}  ,
	\end{equation*}
	and hence we have that
	\begin{equation*}
			\widetilde{\mathcal{D}}_2 \ge 0 .
	\end{equation*}

By $(i)$, $(ii)$, and $(iii)$, we have that all the summands at the left-hand side of \eqref{eq:Fundamental identity higher oreder SBT} are non-negative. Thus, in presence of conditions that force the right-hand side to be non-positive, we have that actually, all the inequalities hold with the equality sign and all the summands at the left-hand side of \eqref{eq:Fundamental identity higher oreder SBT} must be zero; in particular, we have that $L[P]\equiv 0$, and hence, recalling by $(i)$ that $L[P]$ is a {\it spherical detector}, the symmetry follows. We stress that the same conclusion might also be obtained using that \begin{equation*}
	\widetilde{\mathcal{D}}_2 =0 \quad \text{ if and only if } \quad | \na u | \equiv R ;
\end{equation*}
in fact, in this case we have that the solution $u$ to \eqref{eq:Hessian Dirichlet problem} also satisfies the overdetermined problem \eqref{torsionsk_intro}, and hence, by Theorem \ref{Serrin}, $\Om$ must be a ball.

}
	\end{oss}

\begin{proof}[Proof of Theorem \ref{thm:Higher oreder SBT Fundamental Identity}]
	Recalling \eqref{L} and the equation $S_k(D^2u)= \binom{n}{k}$, a direct computation shows that
	\begin{equation*}
		L[P] + \binom{n}{k} k \left[ 1 - \frac{S_{k+1}(D^2 u)}{\binom{n}{k+1}}  \right] = \binom{n}{k} \De u - \frac{n}{n-k} (k+1) S_{k+1} (D^2 u) .
	\end{equation*}
	Hence, by using the divergence theorem, \eqref{skdiv} and \eqref{hksk}, we compute that
	\begin{equation}\label{eq:in passing Identity for SBT}
		\begin{split}
			\int_\Om \left\lbrace L[P]  + \binom{n}{k} k \left[  1 - \frac{S_{k+1}(D^2 u)}{\binom{n}{k+1}}  \right] \right\rbrace \, dx
			= & \binom{n}{k} \int_{\pa\Om} |\na u| d\cH^{n-1} 
			\\
			& - \binom{n}{k} \int_{\pa\Om} H_k |\na u|^{k+1} d\cH^{n-1} .
		\end{split}
	\end{equation}
	Adding and subtracting 
	\begin{equation*}
		\binom{n}{k} \frac{1}{R^k} \int_{\pa\Om} |\na u|^{k+1} d\cH^{n-1},
	\end{equation*}
	the desired identity easily follows.	
\end{proof}

	\begin{cor}\label{cor:Higher order SBT with inequality}
	Let $\Om \subset \RR^n$ be a $C^2$, bounded, connected, open set and let $u\in C^2(\ol{\Om})$ be a solution of \eqref{eq:Hessian Dirichlet problem}. If
	\begin{equation}\label{eq:Ineq SBT}
		H_k \ge \frac{1}{R^k} \text{ on } \pa\Om , \quad \text{where } R \text{ is as in \eqref{def: R new definition}},
	\end{equation}
	then $\Om$ must be a ball of radius $R$ and \eqref{eq:Ineq SBT} is an equality. 
	
	In particular, if 
    $\pa\Om$ has constant $k$-mean curvature, then \eqref{eq:Ineq SBT} holds true and hence $\Om$ must be a ball of radius $R$ and $$H_k=\frac{1}{R^k}=\frac{1}{\hat{R}^k}.$$
	\end{cor}

\begin{oss}\label{rem: R and hat R}
{\rm  We stress that Corollary \ref{cor:Higher order SBT with inequality} remains true if we replace $R$  (which is that defined in \eqref{def: R new definition}) in \eqref{eq:Ineq SBT} with
\begin{equation}
\hat{R} := \frac{n|\Om|}{P(\Om)}  .
\end{equation}
Notice that, for $k=1$, the divergence theorem shows that $R = \hat{R}$. We stress that, for $k>1$, using $R$ provides a more general statement than using $\hat{R}$, since
\begin{equation*}
 R = \frac{1}{P(\Om)} \int_{\pa\Om} |\na u| \, d\cH^{N-1} = \frac{1}{P(\Om)} \int_{\Om} \De u \, dx \ge \frac{n}{P(\Om)} \int_{\Om} \left( \frac{S_k(D^2 u)}{\binom{n}{k}} \right)^{\frac{1}{k}} \, dx =  \frac{n |\Om|}{P(\Om)} = \hat{R}.
\end{equation*}
Here, we used the divergence theorem, Newton's inequality \eqref{newton}, and the equation $S_k (D^2 u) = \binom{n}{k}$.

We also notice that, for $k$-convex domains, the weaker statement with $\hat{R}$ might also be obtained combining the result for $k=1$ in \cite[Theorem 2.2]{MaPoggesiJAM2019} and the inequality
 	\begin{equation}\label{eq: N inequality for k-mean curvature}
 		H_k^{\frac{k-1}{k}} \le H_{k-1} ,
 	\end{equation}
which holds whenever $H_k \ge 0$.
}
\end{oss}

\begin{proof}[Proof of Corollary \ref{cor:Higher order SBT with inequality}]
	The claim immediately follows from \eqref{eq:Fundamental identity higher oreder SBT} and Remark~\ref{rem:Positivity of deficits}, 
noticing that the condition \eqref{eq:Ineq SBT} makes the right-hand side of \eqref{eq:Fundamental identity higher oreder SBT} non-positive.
	
	In order to prove the second claim, let us assume that $H_k$ is constant on $\pa\Om$. As $\pa \Om$ has an elliptic point, $H_k$ must be a positive constant (and hence also $H_i >0$ for any $i \le k$).

	On the one hand, by the divergence theorem, \eqref{minkowski}, and \eqref{eq: N inequality for k-mean curvature}, we compute that	
	\begin{equation*}
		\begin{split}
			n|\Om| H_k = \int_{\pa\Om} H_k \langle x, \nu \rangle \, d\cH^{n-1}
			& = \int_{\pa\Om} H_{k-1} \, d\cH^{n-1}
			\\ & \ge  \int_{\pa\Om} H_{k}^{\frac{k-1}{k}} \, d\cH^{n-1} 
			\\
			& = H_{k}^{\frac{k-1}{k}} P(\Om) .
		\end{split}
	\end{equation*}
	On the other hand,
	by \eqref{eq:Hessian Dirichlet problem}, Newton's inequality, and the divergence theorem, we have that
	\begin{equation*}
		n | \Om | = n \int_{\Om} \left( \frac{S_k( D^2 u)}{ \binom{n}{k} } \right)^{ \frac{1}{k} } dx \le \int_{\Om} \De u \, dx  = \int_{\pa\Om} | \na u | \, d\cH^{n-1} .
	\end{equation*}
	Putting all together, we obtain \eqref{eq:Ineq SBT} and hence the conclusion.

We stress that the second claim of Corollary \ref{cor:Higher order SBT with inequality} returns the higher order Soap Bubble Theorem.
We recall that, whenever $\Om$ is $k-1$-convex and of class $C^{3,1}$, then \cite[Theorem 3.4]{Wang2009} guarantees the existence of a unique solution $u$ of class $C^{3,\al} (\ol{\Om})$ to the Dirichlet problem \eqref{eq:Hessian Dirichlet problem}.
\end{proof}

In passing, in the proof of Theorem \ref{thm:Higher oreder SBT Fundamental Identity} we obtained \eqref{eq:in passing Identity for SBT}, which provides the following interesting corollary involving the overdetermined condition
\begin{equation}\label{eq:alternative overdetermination}
	|\na u| = \frac{1}{H_k^{1/k}} .
\end{equation}

\begin{cor}\label{cor:new overdetermined problem}
Let $\Om \subset \RR^n$ be a $C^2$, bounded, connected, open set, and let $u\in C^2(\ol{\Om})$ be a solution of \eqref{eq:Hessian Dirichlet problem}.
Then the following identity holds true:
\begin{equation}\label{eq:reqritten in passing Identity for SBT}
	\binom{n}{k}^{-1}	\int_\Om  L[P] \, dx +   k \ \widetilde{\mathcal{D}}_1
		= \int_{\pa\Om} \left( 1 - H_k |\na u|^{k} \right) |\na u| d\cH^{n-1} ,
\end{equation}
where 
$\widetilde{\mathcal{D}}_1$ is that defined in Theorem \ref{thm:Higher oreder SBT Fundamental Identity}.

As a consequence, if
\begin{equation*}
	|\na u|^k \, H_k \ge 1  \quad \text{ on } \pa\Om ,
\end{equation*}
then $\Om$ must be a ball. In particular, the same conclusion holds true if the overdetermined condition \eqref{eq:alternative overdetermination} is in force.
\end{cor}
\begin{proof}
The desired conclusion immediately follows by rewriting \eqref{eq:in passing Identity for SBT} as
	\begin{equation}
	\binom{n}{k}^{-1}	\int_\Om  L[P] \, dx  +  k \int_\Om  \left[  1 - \frac{S_{k+1}(D^2 u)}{\binom{n}{k+1}}  \right] \, dx
		= \int_{\pa\Om} \left( 1 
		-  H_k |\na u|^{k} \right) d\cH^{n-1} ,
\end{equation}
and combining it with Remark \ref{rem:Positivity of deficits}.
\end{proof}

The above corollary extends to $k\ge 1$, the result for $k=1$ contained in \cite[(i) of Theorem 2.4]{MaPoggesiJAM2019}.

We are now ready for the

\begin{proof}[Proof of Theorem \ref{thm:symmetry results summary}]
Clearly, it is easy to check that $(i)$ implies each of the other statements in $(ii), (iii), (iv), (v), (vi)$.

The implication $(i) \impliedby (ii)$ follows from either Theorem \ref{Serrin} or Remark \ref{oss:rigidty again}.

The implications $(i) \impliedby (iii)$ and $(i) \impliedby (iv)$ follow from Corollary \ref{cor:Higher order SBT with inequality} and Remark \ref{rem: R and hat R}.

The implication $(i) \impliedby (v)$ follows from Corollary \ref{cor:new overdetermined problem}.

The implication $(i) \impliedby (vi)$ follows by combining \eqref{principal_equality}, \eqref{newton}, and Theorem \ref{L[P]}.
\end{proof}

\subsection{Stability for higher order Soap Bubble type theorems: Proof of Theorems \ref{thm:stability k-SBT one ball} and \ref{thm:Stability with bubbling for k-SBT}}\label{subsec:SBT stability}

\begin{proof}[Proof of Theorem \ref{thm:stability k-SBT one ball}]
Let $u$ be the solution of \eqref{eq:Hessian Dirichlet problem} with $k=1$, that is,
\begin{equation*}
	\begin{cases}
		\De u = n \quad & \text{ in } \Om
		\\
		u = 0 \quad & \text{ on } \pa\Om .
	\end{cases}
\end{equation*}
Recall that, in this case, we have
$$
\hat{R} := \frac{n|\Om|}{P(\Om)} = \frac{1}{P(\Om)} \int_{\pa\Om} |\na u| \, d\cH^{N-1} =: R.
$$

From the fundamental identity \cite[(2.6)]{MaPoggesiJAM2019} (or \eqref{eq:Fundamental identity higher oreder SBT} with $k=1$), we see that
\begin{equation}\label{eq:1 in proof SBT stability L2}
	\nr |\na u| - \hat{R} \nr_{L^2(\pa\Om)}^2 \le M^2 \, \hat{R} \int_{\pa\Om} \left(\frac{1}{\hat{R}} - H \right)^+ d\cH^{n-1} ,
\end{equation}
where $M:= \max_{\ol{\Om}} |\na u|$.

Notice that 
\begin{equation}\label{eq:upper bound for hatR}
\hat{R} = \frac{n|\Om|}{P(\Om)} \le \left( \frac{|\Om|}{\omega_n} \right)^{1/n} \le \frac{d_\Om}{2},
\end{equation}
which follows from the isoperimetric inequality.

Since $H_k \ge 0$, by Newton's inequality we have that
\begin{equation}\label{eq:NewtonIneq for H_k complete chain}
	H_k^{1/k} \le H_{k-1}^{1/(k-1)} \le \dots \le H_1=H ,
\end{equation}
and hence 
\begin{equation}\label{eq:inequality RHS SBT FROM CLASSICAL TO K VIA NEWTON}
	\left(\frac{1}{\hat{R}} - H \right)^+ \le \left(\frac{1}{\hat{R}} - H_k^{1/k} \right)^+ \le \hat{R}^{k-1} \left(\frac{1}{\hat{R}^k} - H_k \right)^+ ,	
\end{equation}
where the last inequality can be easily verified noting that, for any $0 \le a \le b$, we have that
\begin{equation*}
	b^k - a^k \ge b^{k-1} (b-a) .
\end{equation*}
Combining \eqref{eq:1 in proof SBT stability L2} and \eqref{eq:inequality RHS SBT FROM CLASSICAL TO K VIA NEWTON} we obtain
\begin{equation}\label{eq:2 in proof SBT stability L2}
	\nr |\na u| - \hat{R} \nr_{L^2(\pa\Om)}^2 \le M^2 \, \hat{R} \int_{\pa\Om} \left(\frac{1}{\hat{R}} - H_k^{1/k} \right)^+ d\cH^{n-1} 
	\le  M^2 \, \hat{R}^k \int_{\pa\Om} \left(\frac{1}{\hat{R}^k} - H_k \right)^+ d\cH^{n-1}. 
\end{equation}

Since $H_k \ge 0$, by \eqref{eq:NewtonIneq for H_k complete chain} we also have that $H\ge 0$, that is, $\Om$ is a mean-convex domain;
thus, $M$ can be estimated as follows (see, e.g., \cite[Lemma 2.2]{Poggesi_JMPA2025} and \cite[Lemma 2.2]{MaPog_HotSpotsMA2022}):
\begin{equation}\label{eq:upper bound M for k=1 and mean convex}
M^2 = \max_{\ol{\Om}} |\na u|^2 \le 2 n \, \max_{\ol{\Om}} (-u) \le n \, d_\Om^2 .
\end{equation}
Using \cite[(i) of Theorem 5.1]{PaPoRo_MZ2024}, noting that 
\begin{equation*}
\nr |\na u|^2 - \hat{R}^2 \nr_{L^2(\pa\Om)} \le \left( \max_{\ol{\Om}} |\na u| + \hat{R} \right) \nr |\na u| - \hat{R} \nr_{L^2(\pa\Om)},
\end{equation*}
and recalling \eqref{eq:upper bound for hatR} and \eqref{eq:upper bound M for k=1 and mean convex}, we obtain that  
\begin{equation*}
	\nr |x-z| - \hat{R} \nr_{L^2(\pa\Om)} \le C(n, r_i, d_\Om) \nr |\na u| - \hat{R} \nr_{L^2(\pa\Om)} , 
\end{equation*}
which combined with \eqref{eq:2 in proof SBT stability L2} (and again \eqref{eq:upper bound for hatR} and \eqref{eq:upper bound M for k=1 and mean convex}) leads to \eqref{eq:L2 stability for k-SBT}.

The Hausdorff stability in \eqref{eq:stability HK rhoei} follows immediately by combining \eqref{eq:inequality RHS SBT FROM CLASSICAL TO K VIA NEWTON} and \cite[Theorem 3.7]{MaPogNearlyCVPDE2020} (replacing the application of \cite[Theorem 2.10]{MaPogNearlyCVPDE2020} with that of \cite[Theorem 3.4]{MaPog_MinE2023}, and using that $\nr \na h \nr_{L^\infty(\Om)}$ appearing in the latter can be estimated by means of $\nr \na h \nr_{L^\infty(\Om)} \le M + d_\Om$ and \eqref{eq:upper bound M for k=1 and mean convex}).
\end{proof}

We conclude by providing the 
\begin{proof}[Proof of Theorem \ref{thm:Stability with bubbling for k-SBT}]
	The desired result follows by combining \cite[Theorem 1.4]{Poggesi_JMPA2025}
    and \eqref{eq:inequality RHS SBT FROM CLASSICAL TO K VIA NEWTON}.
\end{proof}

\section*{\centering Appendix }\sectionmark{Appendix}
\appendix

\section{Sobolev-Poincar\'e type inequalities for solutions of homogeneous uniformly elliptic equations and general interpolation inequalities} \sectionmark{Appendix} \label{appendix}

In this Appendix, we collect some results which are useful in the study of the quantitative stability of the symmetry results for the Hessian Dirichlet problem \eqref{eq:Hessian Dirichlet problem}. 

We start providing a weighted Sobolev-Poincar\'e-type inequality for solutions of homogeneous uniformly elliptic equations, with a proof that allows to obtain explicit estimates of the constants. Since the result may be of independent interest we will present it for general uniformly elliptic linear operator in divergence form (even though in our application we will apply it to the operator $L$ defined in \eqref{ell}) and within the general class of John domains.

\begin{definizione} 
A domain $\Om \subset \mathbb{R}^n$ is a $b_0$-John domain if each pair of distinct points $x_1$ and $x_2$ in $\Om$ can be joined by a curve $\gamma: [0,1] \to \Om$ such that
\begin{equation*}
	\de_{\pa\Om} ( \gamma ( t ) ) \ge \frac{ \max \left\lbrace \gamma(t) - x_1 , \gamma(t) - x_2 \right\rbrace }{b_0} .
\end{equation*}
\end{definizione}

\begin{teorema}\label{thm:Poincare for solutions}
Let $\Om\subset \RR^n$ be a bounded $b_0$-John domain and let $x_0\in\Om$. Let $\mathcal{L}$ be a uniformly elliptic linear operator in divergence form defined by
\begin{equation*}
\mathcal{L} f := \mathrm{div} \left[A(x) \na f \right], \quad \text{ for } x\in\Om,
\end{equation*}
where $A(x)$ is a symmetric $n \times n$ matrix whose entries are bounded measurable functions on $\Om$ satisfying the uniform ellipticity condition
\begin{equation*}
	\la | \xi |^2 \le \langle A(x) \xi, \xi \rangle \le \La | \xi |^2 \quad \text{ for any } x \in \Om, \, \xi \in \RR^n ,
\end{equation*}
where $\la$ and $\La$ are positive constants.

Let $r$, $p$, and $\al$ be three numbers such that, either
\begin{equation}\label{eq:p and r as in classical Poincare}
	r=p \in \left[1 , \infty \right), \quad \al=0 ,
\end{equation}
or
\begin{equation}\label{eq:p and r as in Sobolev-Poincare}
0 \le \al \le 1 , \quad p(1-\al) < n , \quad 	1 \le p \le r \le \frac{np}{n-p(1-\al)} .
\end{equation}

Then, there exists a positive constant $C_{r,p,\al}(\Om,x_0)$ such that
\begin{equation}\label{eq:Sobolev Poincare ineq for general solution f}
	\nr f \nr_{L^r(\Om)} \le C_{r, p,\al}(\Om,x_0) \nr \de_{\pa\Om}^\al \na f \nr_{L^p(\Om)} ,
\end{equation}
for every function $f$ -- with $f \in L^1_{loc}(\Om)$ and $\de_{\pa\Om}^\al \na f \in L^p(\Om)$, satisfying 
$$\mathcal{L} f=0 \text{ in } \Om, \, \,  f(x_0) = 0, \quad x_0\in\Omega.$$
Moreover, the constant $C_{r,p,\al}(\Om,x_0)$ can be estimated (from above) in terms of $n, r, p, \al$, $\la$, $\La$, (upper bounds on) $b_0$ and $d_\Om$, and (a lower bound on) $\de_{\pa\Om}(x_0)$ only.
\end{teorema}

The above inequality was first noticed by Ziemer \cite{ziemer} for harmonic functions (i.e., in the case of the classical Laplacian $\mathcal{L}:=\De$) in the case \eqref{eq:p and r as in classical Poincare}. The extension to homogeneous linear equations in divergence form with locally bounded coefficients in the weighted setting with $0 \le \al\le 1$ and $r=p \in \left[1 , \infty \right)$
appeared in Example 2.5 of \cite{straube} (where it is proved for $C^{0,\al}$ domains). Their proof is based on a standard compactness argument (which dates back at least to Morrey \cite{Morrey1966}), and hence does not provide explicit estimates for the constant $C_{p, p,\al}(\Om,x_0)$ appearing in the inequality. 
Notice that \eqref{eq:p and r as in Sobolev-Poincare}-\eqref{eq:Sobolev Poincare ineq for general solution f} provides the improvement of the Poincar\'e-type inequality \cite[Example 2.5]{straube} to its corresponding Sobolev-Poincar\'e-type inequality, which holds true when $p(1-\al)<n$, and for $b_0$-John domains.
Moreover, our proof
provides  explicit estimates, leveraging the proof of the weighted Sobolev-Poincar\'e type inequalities (for general Sobolev functions) obtained by Hurri-Syrjanen \cite{Hurri_1988,Hurri_PAMS1994} and the generalized mean value theorem for linear uniformly elliptic equations discovered by Caffarelli~\cite{Caffarelli_FermiLectures1998}. 

\begin{proof}[Proof of Theorem \ref{thm:Poincare for solutions}]
Theorem 1.3 in \cite{Hurri_1988} -- resp., Theorem 1.4 in \cite{Hurri_PAMS1994} -- shows that, if $r,p,\al$ are as in \eqref{eq:p and r as in classical Poincare} -- resp., \eqref{eq:p and r as in Sobolev-Poincare} -- then the inequality
\begin{equation}\label{eq:Poincare from HS}
	\nr f - f_{r, \Om} \nr_{L^r(\Om)} \le \Tilde{C}_{r, p,\al}(\Om) \nr \de_{\pa\Om}^\al \na f \nr_{L^p(\Om)}	
\end{equation}
holds true for every $f \in L^1_{loc}(\Om)$ such that $\de_{\pa\Om}^\al \na f \in L^p(\Om)$. Here, $f_{r, \Om}$ denotes the ``$r$-mean'' of $f$ in $\Om$, which is defined (see \cite{IshMaWa_CVPDE2017}) as the unique minimizer of the problem
\begin{equation*}
	\inf_{\la \in \mathbb{R}} \nr f - \la \nr_{L^r(\Om)} .
\end{equation*} 
It is easy to check that in the case $r=2$ this recovers the classical mean $f_\Om$ of $f$ in $\Om$, that is,
$$
f_{2,\Om} = f_{\Om}:= \frac{1}{|\Om|}\int_{\Om} f \, dx .
$$

We stress that the proofs in \cite{Hurri_1988,Hurri_PAMS1994} (see, also \cite[Remark 2.4]{MaPogNearlyCVPDE2020} for details) allow to obtain explicit upper bounds for $\Tilde{C}_{r, p,\al}(\Om)$ depending only on $n,r,p,\al,$ and upper bounds for $b_0$ and the diameter $d_\Om$. For instance, in the case \eqref{eq:p and r as in Sobolev-Poincare}, we have that
\begin{equation}\label{eq:constant in HS Poincare}
	\Tilde{C}_{r, p,\al}(\Om) \le k_{n,r,p,\al} \,  b_0^n \, |\Om|^{\frac{1-\al}{n} - \frac{1}{p} + \frac{1}{r}}.
\end{equation}

In order to manipulate the left-hand side of \eqref{eq:Poincare from HS}, we notice that (see, e.g., \cite[Lemma 3.3]{Poggesi_AMPA2024}) for any $F \subset \Om$ with positive measure and for any $\la\in \RR$ we have that
\begin{equation}\label{eq: easy inequality mean}
	\nr f - f_{F} \nr_{L^r(\Om)} \le
	\left[ 1+ \left(\frac{|\Om|}{|F|}\right)^{\frac{1}{r}} \right] \nr f - \la \nr_{L^r(\Om)}
	\le 2 \left(\frac{|\Om|}{|F|}\right)^{\frac{1}{r}} \nr f - \la \nr_{L^r(\Om)} .
\end{equation} 
Here, in the second inequality we used that $|\Om|/|F| \ge 1$ as $F \subset \Om$.

Now, if $f$ satisfies $\mathcal{L} f=0$ in $\Om$, then, for $x_0 \in \Om$, the generalized mean value theorem proved by Caffarelli (see the Remark on page 9 in \cite{Caffarelli_FermiLectures1998} and Theorem 6.3 in \cite{BlankZheng_2015} for a careful full proof) ensures that there exist two constants $c$, $C$ that only depend on $n$, $\la$, and $\La$, and an increasing family of domains $D_r (x_0)$ which satisfy the properties
\begin{enumerate}[label=(\roman*)]
	\item $B_{c r}(x_0) \subset D_r (x_0) \subset B_{Cr}(x_0)$ ,
	\item $f(x_0) = f_{D_r(x_0)}$ for any $0 < r < \de_{\pa\Om} (x_0)/C$.
\end{enumerate}
Using \eqref{eq: easy inequality mean} with $F:= D_\frac{\de_{\pa\Om} (x_0)}{2C} (x_0)$ and $\la:= f_{r,\Om}$, we thus obtain that
\begin{equation}\label{eq:proof MVT + easy ineq}
\begin{split}
\nr f - f (x_0) \nr_{L^r(\Om)} 
& \le
2\, \left(\frac{|\Om|}{ \left| D_{\frac{\de_{\pa\Om} (x_0)}{2C}} (x_0) \right| }\right)^{\frac{1}{r}} \nr f - f_{r,\Om} \nr_{L^r(\Om)}
\\
& \le 
\left( \frac{ 2^{n+r}}{\omega_n} \right)^{\frac{1}{r}} \, \left( \frac{C}{c} \right)^{\frac{n}{r}}  \frac{|\Om|^{\frac{1}{r}}}{ \left( \de_{\pa\Om} (x_0) \right)^{\frac{n}{r}} } \,   \nr f - f_{r,\Om} \nr_{L^r(\Om)} ,
\end{split}
\end{equation}
where the second inequality follows by using that
$$
\left| D_{\frac{\de_{\pa\Om} (x_0)}{2C}} (x_0) \right| \ge \omega_n \left( \frac{c\, \de_{\pa\Om} (x_0)}{2C} \right)^n  \text{ since } D_\frac{\de_{\pa\Om} (x_0)}{2C} (x_0) \supset B_{\frac{c\, \de_{\pa\Om} (x_0)}{2C}} (x_0) .
$$

Combining \eqref{eq:Poincare from HS} and \eqref{eq:proof MVT + easy ineq}, we easily obtain that \eqref{eq:Sobolev Poincare ineq for general solution f} holds true for every function $f$  as in the assumptions of the Theorem.
 Moreover, the constant $C_{r,p,\al}(\Om,x_0)$ can be estimated from above in terms of $n, r, p, \al$, $\la$, $\La$,  $b_0$ and $d_\Om$, and (a lower bound on) $\de_{\pa\Om}(x_0)$ only.
For instance, in the case \eqref{eq:p and r as in Sobolev-Poincare}	(which is the one that will be used in the present paper), recalling \eqref{eq:Poincare from HS}, \eqref{eq:constant in HS Poincare}, and \eqref{eq:proof MVT + easy ineq} gives that
\begin{equation}\label{eq:explicit upper bound in Sobolev Poincare}
C_{r, p,\al}(\Om,x_0) \le k_{n,r,p,\al} \, \, \left( \frac{C}{c} \right)^{\frac{n}{r}} \left( \frac{b_0}{\de_{\pa\Om} (x_0)^\frac{1}{r}  }\right)^n \, |\Om|^{\frac{1-\al}{n} - \frac{1}{p} + \frac{2}{r}} ,	
\end{equation}
up to renaming the constant $k_{n,r,p,\al}$ (which only depends on $n,r,p,\al$). Here, $c$ and $C$, which only depend on $n$, $\la$, and $\La$, are exactly those appearing in \eqref{eq:proof MVT + easy ineq} and in the generalized mean value theorem \cite[Theorem 6.3]{BlankZheng_2015}.
\end{proof}

In order to prove our stability results, we also need the following interpolation inequality obtained in \cite{MaPog_MinE2023,MaPog_Ball_CVPDE2024}. 

As in \cite{MaPog_MinE2023}, we adopt the following definition of uniform interior cone condition. A bounded domain $E \subset \mathbb{R}^n$ satisfies the $(\theta , a)$-uniform interior cone condition if, for any $x\in \ol{E}$, there exists a cone $\mathcal{C}_x$ with height $a$ and opening width $\theta$, such that $\mathcal{C}_x \subset E$ and $\mathcal{C}_x \cap \pa E = \left\lbrace x \right\rbrace$.

Moreover, we say that a bounded domain $\Omega \subset \mathbb{R}^n$ is a $C_p$-Poincar\'e domain if 
\begin{equation*}
	\nr f - f_{\Omega} \nr_{L^p(\Omega)} \le C_{p} \nr \na f \nr_{L^p(\Omega)} \quad \text{ for any } f \in W^{1,p}(\Omega).
\end{equation*}

Taking into account \cite[Appendix A.2]{MaPog_Ball_CVPDE2024}, the results in \cite[Theorems 2.4 and 2.7]{MaPog_MinE2023} read as follows.

\begin{teorema}[{\cite[Theorems 2.4 and 2.7]{MaPog_MinE2023}} and {\cite[Appendix A.2]{MaPog_Ball_CVPDE2024}}]
	\label{thm:Interpolation from MPMine + MPCVPDE}
Let $1\le p < q \le \infty$. Let $\Omega \subset \RR^n$ be a bounded $C_p$-Poincar\'e domain satisfying the $(\theta , a)$-uniform interior cone condition. 

Then, setting
\begin{equation*}
	\al_{p,q} := \frac{p(q-n)}{n(q-p)},
\end{equation*}
we have that
\begin{equation}\label{eq:general interpolation inequality with Poincare constant}
	\max_{\ol{\Omega}} f - \min_{\ol{\Omega}} f \le C 
	\begin{cases}
		\nr \na f \nr_{L^p(\Omega)} \, & \text{ for any } f\in W^{1,p}(\Omega), \,  \text{ if } p>n ,
		\\
		\nr \na f \nr_{L^n(\Omega)} \log \left(e \frac{ |\Omega|^{\frac{1}{n} - \frac{1}{q}} \nr \na f \nr_{L^q(\Omega)} }{ \nr \na f \nr_{L^n(\Omega)} }\right) \, & \text{ for any } f\in W^{1,q}(\Omega), \,  \text{ if } p=n ,
		\\
		\nr \na f \nr_{L^q(\Om)}^{1-\al_{p,q}}	\nr \na f \nr_{L^p(\Om)}^{\al_{p,q}} \,  & \text{ for any } f\in W^{1,q}(\Omega), \,  \text{ if } 1 \le p < n ,
	\end{cases}
\end{equation}
where $C$ is a constant that can be explicitly estimated in terms of $n, p, q,$ $C_p$, $\te, a,$ and the diameter $d_\Om$. 
\end{teorema}
\begin{oss}\label{rem:(iii)if r_i then b_0, theta, a, and C_p can be estimated (iv) for Poincare with L and r_i.}
{\rm (i) We recall that if a bounded domain $\Omega$ is a $b_0$-John domain then it is a $C_p$-Poincar\'e and the constant $C_p$ can be explicitly estimated in terms of $n, p , b_0,$ and $d_\Om$ (see, e.g., \cite{Bojarski1988,Hurri_1988} and \cite[(ii) of Remark 2.4]{MaPogNearlyCVPDE2020}).
	
	(ii) If $\Omega\subset\RR^n$ is a bounded domain satisfying the uniform interior sphere condition with radius $r_i$, then it is a $b_0$-John domain with
	$b_0 \le d_\Omega/r_i$ and it satisfies the $(\theta , a)$-uniform interior cone condition with $\theta:=\pi/4$ and $a:=r_i$, as it can be easily proved by the definitions.
	 
	(iii) Combining (i) and (ii) of the present remark, we have that, if $\Omega$ satisfies the uniform interior sphere condition with radius $r_i$, then the constant $C$ in \eqref{eq:general interpolation inequality with Poincare constant} can be explicitly estimated only in terms of $n, p, q,$ $r_i$, and $d_\Om$.
	
	(iv) Regarding Theorem \ref{thm:Poincare for solutions}, in the particular case where $\mathcal{L}:=L$ (where $L$ is that defined in \eqref{ell}) and $\Om$ satisfies the uniform interior sphere condition with radius $r_i$, recalling \eqref{eq:explicit uniform ellipticity 1}-\eqref{eq:explicit uniform ellipticity 2} and using that $b_0 \le d_\Om/r_i$, we have that the constant $C_{r,p,\al}(\Om, x_0)$ in \eqref{eq:Sobolev Poincare ineq for general solution f} can be explicitly estimated only in terms of $n,r,p,\al,$ $r_i$, $d_\Om$, and (a lower bound on) $\de_{\pa\Om}(x_0)$. In particular, an explicit estimate in the case \eqref{eq:p and r as in Sobolev-Poincare} can be deduced from \eqref{eq:explicit upper bound in Sobolev Poincare}.
}	
\end{oss}


\vspace{2mm}
\section*{Acknowledgements}
The first three authors have been partially supported by GNAMPA group of INdAM. \\
Nunzia Gavitone was supported by the Project MUR PRIN-PNRR 2022:  "Linear and Nonlinear PDE’S: New directions and Applications", P2022YFA.\\
During the preparation of this work, Alba Lia Masiello was partially supported by  the Project PRIN 2022, 20229M52AS: "Partial differential equations and related geometric-functional
inequalities," CUP:E53D23005540006. and by  "INdAM - GNAMPA Project"  CUP E5324001950001
\\
	Gloria Paoli was supported by "INdAM - GNAMPA Project", codice CUP E5324001950001
    and  by the Project PRIN 2022 PNRR:  "A sustainable and trusted Transfer Learning platform for Edge Intelligence (STRUDEL)", CUP E53D23016390001, in the framework of European Union - Next Generation EU program
\\
Giorgio Poggesi is supported by the Australian Research Council (ARC) Discovery Early Career Researcher Award (DECRA) DE230100954 “Partial Differential Equations: geometric aspects and applications” and is a member of the Australian Mathematical Society (AustMS).

\bibliographystyle{plain}
\bibliography{biblio}
\end{document}